\documentclass{article}

\usepackage[utf8]{inputenc}
\usepackage[T1]{fontenc}
\usepackage{authblk}
\usepackage{amssymb}
\usepackage{amsmath}
\usepackage{amsthm}
\usepackage{bbold}
\usepackage{physics}
\usepackage{paralist}
\usepackage[misc]{ifsym}
\usepackage{epsfig}
\usepackage{euscript}
\usepackage{enumitem}

\usepackage{tikz,pgf}
\usepackage{geometry}
\usetikzlibrary{arrows}
\usepackage{pgfplots}
\pgfplotsset{compat=1.5}
\definecolor{BLUE}{rgb}{0.30196078431372547,0.30196078431372547,1}
\definecolor{RED}{rgb}{1,0,0}

\newcommand{\intint}[2]{\left[\!\left[#1,#2\right]\!\right]}

\newcommand{\1}{\mathbb{1}}

\newcommand{\C}{\mathcal{C}}

\newcommand{\F}{\mathcal{F}}

\renewcommand{\hbar}{\overline{h}}

\newcommand{\M}{\mathcal{M}}
\newcommand{\N}{\mathbb{N}}

\renewcommand{\P}{\mathcal{P}}

\newcommand{\R}{\mathbb{R}}

\newcommand{\T}{\mathcal{T}}
\newcommand{\V}{\mathcal{V}}
\newcommand{\W}{\mathcal{W}}

\newcommand{\Supp}{\mathrm{Supp}}
\newcommand{\dsp}{\displaystyle}

\DeclareMathOperator*{\esssup}{ess\,sup}
\DeclareMathOperator*{\essinf}{ess\,inf}

\theoremstyle{plain}
\newtheorem{theorem}{Theorem}[section]
\newtheorem{proposition}[theorem]{Proposition}
\newtheorem{lemma}[theorem]{Lemma}

\newtheorem{remark}[theorem]{Remark}
\newtheorem{assumption}{Assumption}

\theoremstyle{definition}
\newtheorem{definition}{Definition}[section]

\theoremstyle{remark}

\usepackage[justification=centering]{caption}
\DeclareCaptionFormat{sanslabel}{#3}%

\usepackage{hyperref}
\hypersetup{                    
    colorlinks=true,                
    breaklinks=true,                
    urlcolor= blue,                 
    linkcolor= red,                
    citecolor= blue                
    }

\title{Mean-Field limit of the non-exchangeable Cucker-Dong model}

\author[1]{Nathalie Ayi}
\author[2]{Adrien Cotil}
\author[3]{Fanny Delebecque}

\affil[1]{\footnotesize Sorbonne Université, Université Paris Cité, CNRS, Inria, Laboratoire Jacques-Louis Lions (LJLL), Institut Universitaire de France, F-75005 Paris, France \vspace{0.15cm}}
\affil[2]{\footnotesize  Sorbonne Université, Université Paris Cité, CNRS, Laboratoire Jacques-Louis Lions (LJLL), F-75005 Paris, France \vspace{0.15cm}}
\affil[3]{\footnotesize Université de Toulouse, Institut de Mathematiques de Toulouse, F-31062 Toulouse, France.}

\date{}

\begin{document}

\maketitle

\begin{abstract}
In this article, we examine the mean-field limit of the non-exchangeable Cucker–Dong model. This model corresponds to a biologically more realistic version of the classic Cucker-Smale model, which is used to describe the alignment phenomenon in large animal groups. In addition to alignment forces, the non-exchangeable Cucker–Dong model integrates attraction/repulsion forces and network-structured interactions. In order to enable convergence towards a flocking profile, the attraction/repulsion forces are weighted by a second-order coefficient called the alignment measure, which is smaller when individuals are more aligned overall. Deriving the mean-field limit of this model relies on a new stability result that is in agreement with with both the second-order nature of the alignment measure and the non-exchangeability induced by the graph-dependent interactions.
\end{abstract}

\section{Introduction}

Models of collective dynamics in animal groups have attracted considerable attention over the past decades, motivated by striking natural phenomena such as the flocking of birds, the herding of sheep or the swarming of fish   \cite{Aoki82, Ballerini08, Lopez12}. A huge literature has been devoted to understanding how simple local interaction rules between individuals can give rise to coherent large-scale patterns. Among the most prominent examples is the Cucker–Smale model, introduced in \cite{cucker2007emergent,cucker2}, which describes a system of interacting agents whose velocities tend to align through distance-dependent communication weights. The corresponding particle system reads:
\begin{equation*}\label{eq:CS}
    \left\{\begin{aligned}
        &\dot{x}_i=v_i, \\
        &\dot{v}_i=\dfrac{1}{N}\sum_{j=1}^N\psi(x_j-x_i)(v_j-v_i) ,
    \end{aligned}\right.
\end{equation*}
where $x_i(t),v_i(t)\in\R^d$, with $d\geq1$, are the position and the velocity of agent $i\in\intint{1}{N}:=\{1,\dots,N\}$ at time $t\geq0$, and $\psi:\R^+\to\R$ is the communication weight function. This model has been extensively studied from a rigorous mathematical viewpoint, in particular with respect to its long-time behavior, and many works have characterized sufficient conditions under which flocking occurs or fails to occur, see for instance \cite{ahn,cucker2007emergent,ha2009simple,ha2008particle}. Roughly speaking, flocking refers to the emergence of a common asymptotic velocity for all agents together with a uniform control of the spatial diameter of the group.\\

When considering such systems, we may study their long-time dynamics, as mentioned above, but we may also, as we do here, address the question of their large-population limit. Indeed, when the number of agents $N$ is large, the discrete system can be approximated by an alternative description, replacing the dynamics of $N$ interacting agents by a single partial differential equation (PDE) of Vlasov-type for the measure $\mu_t(x,v)$, corresponding to the density of particles at $x$ with velocity $v$ at time $t$. The rigorous justification of this approximation is commonly referred to as the derivation of the mean-field limit. In the context of deterministic initialization, this proof most often involves proving the stability of the associated PDE, from which the mean-field limit follows as a direct corollary. The first rigorous mean-field results were obtained in the framework of the kinetic theory of gases \cite{braun1977,dobrushin1972}. \\ 

However, one of the assumptions underlying the previously mentioned mean-field limits, which we have not yet made explicit, is the  exchangeability of the particles in the system. In contrast, in models of collective dynamics for biological or social systems, such as cattle or sheep herds, interactions between individuals are often heterogeneous, reflecting social structure or hierarchy. Thus, to account for this feature, we introduce an interaction matrix $A$, which encodes how the social structure influences the motion of the agents in the model. From this perspective, the system can be reinterpreted as being posed on a graph, whose weight matrix is given by $A$.  Flocking results have already been established in such non-exchangeable settings (see \cite{cotil2022flocking,cotil2023} for instance). \\

Stepping back for a moment from the specific Cucker–Smale framework and consider general first-order interacting particle systems, a large body of work has been devoted to the rigorous derivation of non-exchangeable mean-field limits. Recent progress in graph limit theory, building on the seminal work of Lovász \cite{Lovasz12}, has provided powerful tools to address this problem  (see, e.g., \cite{KaliuzhnyiMedvedev18,ChibaMedvedev19,Kuehn20,JabinPoyatoSoler21,KuehnXu22}), as well as  the recent survey \cite{AyiPouradierDuteil24} and the lecture notes \cite{Ayi2026}). The key idea consists in introducing a measure  $\mu_t^\xi(x)$, corresponding, this time, to the probability of finding an agent with identity $\xi$ and position $x$ at time $t$. This enriched description makes it possible to close the limit equation and to rigorously derive a non-exchangeable mean-field limit.\\

In addition to heterogeneous alignment forces, classical modeling of animal group behavior includes long-distance attraction forces and short-distance repulsion forces. These interactions are collectively known as the first principles of swarming \cite{carrillo2017review} and were introduced in the seminal works of Aoki \cite{Aoki82} and Reynolds \cite{Reynolds87}. The addition of such forces to the alignment one in the Cucker-Smale model, in the context of symmetric interactions, leads to the so-called Cucker-Dong model introduced in \cite{cucker2011general}. We propose in this article to study the mean-field limit of a non-exchangeable version of the Cucker-Dong model, given by the following ordinary differential equation (ODE):
\begin{equation}\label{eq:CD}
    \left\{\begin{aligned}
        &\dot{x}_i=v_i, \\
        &\dot{v}_i=\dfrac{1}{N}\sum_{j=1}^NA_{ij}^NK(x_j-x_i,v_j-v_i) + \dfrac{\sigma(v)^{2\alpha-1}}{N}\sum_{j=1}^N\phi(\norm{x_i-x_j})(x_i-x_j),
    \end{aligned}\right.
\end{equation}
where $x_i(t),v_i(t)\in\R^d$ are the position and velocity of agent $i\in\intint{1}{N}$ at time $t\geq0$,  $A^N=(A_{ij}^N)_{1 \leq i,j \leq N }\in\R_+^{N\times N}$ is the interaction matrix, $K:\R^{2d}\to\R^{d}$ is the communication kernel and $\phi:\R^+\to\R$ is the collision avoidance function. In the case where $K$ is given by $K(x,v)=\psi(\norm{x})\Gamma(v)$  and $A_{ij}^N=1$, we recover the Cucker-Dong model. Additionally, if $\Gamma(v)=v$ and $\phi=0$, we recover the standard Cucker-Smale model. For the sake of generality and presentation, we will study \eqref{eq:CD} with a general kernel $K$. The quantity $\sigma(v)$, that we call the alignment measure, is given by
$$
\sigma(v)^2:=\dfrac{1}{2N^2}\sum_{i,j}\norm{v_i-v_j}^2=\dfrac{1}{N}\sum_{i=1}^N\norm{v_i-\dfrac{1}{N}\sum_{j=1}^Nv_j}^2=\dfrac{1}{N}\sum_{i=1}^N\norm{v_i}^2-\norm{\dfrac{1}{N}\sum_{i=1}^Nv_i}^2,
$$
where $\norm{\cdot}$ denotes the Euclidean norm on $\R^d$ and $\alpha\geq1$.\\

The perspective that motivates the present work is a particular interest in models that explicitly avoid collisions between animals, a feature that is biologically relevant. Several mechanisms have been proposed to prevent collisions,  including the use of a singular communication weights \cite{ahn,Carrillo2017,Maupoux23} or a singular collision avoidance function \cite{cucker2011general}. In all cases, the forces need to be non integrable at the origin in order to avoid collisions, which leads to so-called strongly singular kinetic models. However, existing mean-field results for singular weights in are restricted to integrable forces that do not reach the collision-avoiding regime. Results allowing stronger singularities are available at the price of introducing $N$-dependent cut-offs with a suitable scaling \cite{pickl}. To the best of our knowledge, a mean-field limit for collision-avoiding models with $\beta>1$ and without cut-off has not yet been established. \\

For the Cucker-Dong model, a first attempt of deriving its mean-field limit with regular interactions was proposed in \cite{yang2014}. However, the main result claimed  does not hold in general. In particular, we provide in this paper an explicit counterexample showing that  \eqref{eq:CD} is not stable according to the  Wasserstein distance of order 1 and thus that the conclusion of their theorem cannot be valid under the stated assumptions. The aim of the present work is therefore to rigorously establish the mean-field limit for the non-exchangeable Cucker–Dong model \eqref{eq:CD}, in the context of locally Lipschitz interactions.  More specifically, we prove that \eqref{eq:CD} is stable according to the Wasserstein distance of order $2$, whose use is not standard in the study of the stability of Cucker-Smale-type models. Roughly speaking, this is due to the presence of the  second-order term $\sigma(v)$ whose control requires the use of a distance of the same order for small values of $\alpha$.\\

In this article, we obtain a rigorous proof of the non-exchangeable mean-field limit and show the convergence, as $N$ goes to $\infty$, of the particle system towards the associated Vlasov-type equation, which we refer to as the kinetic Cucker–Dong equation:
\begin{equation}\label{eq:KCD1}
\partial_t \mu_t^\xi (x,v)+v \cdot \nabla_x \mu_t^\xi (x,v) + \nabla_v \cdot  \left(F_a[\mu_t](\xi,x,v)\,\mu_t^\xi (x,v)\right)=0, \quad t \in [0,T],\, \xi \in [0,1],\,x\in \mathbb{R}^d,\, v \in \mathbb{R}^d,
\end{equation}
where 
\begin{multline} \label{eq:force}
F_a[\mu_t](\xi,x,v) = \iint_{ [0,1] \times \R^{2d}}a(\xi,\zeta)K(y-x,w-v)\,\mu_t^\zeta(dy,dw) \nu(d\zeta) \\
    +\sigma_v(\mu)^{2\alpha-1}\iint_{[0,1]\times \R^{2d}}\phi(\norm{x-y})(x-y)\,\mu_t^\zeta(dy,dw) \nu(d\zeta),
\end{multline}
with  $a$ a quantity associated with the limit of the sequence of graphs with weight matrix $A^N$, whose precise construction will be specified later,  and
with the continuous alignment measure $\sigma_v(\mu)$ being given by
\begin{equation}\label{eq:alignment_measure_cont}
\sigma_v(\mu)^2 :=\iint_{[0,1] \times \R^{2d}}\norm{w}^2\,\mu_t^\zeta(dy,dw) \nu(d\zeta) -\norm{\iint_{[0,1]\times \R^{2d}}w\,\mu_t^\zeta(dy,dw) \nu(d\zeta)}^2,
\end{equation}
where $T > 0$ denotes a constant fixed throughout the paper and the measure $\nu$ belongs to the space of probability measures on $[0,1]$ denoted $\P([0,1])$. In this framework, the measure $\nu$ describes the distribution of labels in the population. \\

The rest of the paper is organized as follows. In Section~\ref{sec:def_results}, we first introduce the general framework on which our analysis is based. This includes elements of graph theory in Section~\ref{sec:graph} and the related functional spaces in Section~\ref{sec:spaces}. The assumptions and the main result are stated in Section~\ref{sec:main_result}. Section~\ref{sec:well_posedness} is dedicated to proving the global well-posedness of \eqref{eq:KCD1}. In Section~\ref{sec:stability}, we establish some stability estimates  with respect to the Wasserstein distance of order $2$ and to the cut distance introduced in Section~\ref{sec:graph}. The main result of the paper is the rigorous proof of the mean-field limit, which is established in  Section~\ref{sec:mean-field}. Finally, in Section~\ref{sec:conclusion}, we conclude our work by providing, in particular, an example demonstrating that the method used in this paper cannot yield a stability result for the Wasserstein distance of order $1$.

\section{Setting and Main Results}\label{sec:def_results}

This section introduces the main definitions, assumptions and results of the paper. In Section~\ref{sec:graph}, we present the graph-theoretic notions used throughout. In Section~\ref{sec:spaces} we detail the functional spaces use to define the notion of solution of \eqref{eq:KCD1}. Then, in Section~\ref{sec:main_result}, we state the main results in a sole theorem that will be expanded upon latter.

\subsection{Preliminaries on graph theory}\label{sec:graph}

We briefly recall a few notions from graph theory and graph limit theory that will be used throughout the paper. We refer the reader to \cite{Lovasz12} for further details.\\

We begin by introducing the basic notion of a graph.
\begin{definition}[A graph] A  graph is a pair $G=(V(G),E(G))$ with $V(G)$ a finite set and $E(G) \subset V(G)\times V(G)$. Elements in $V(G)$ are called nodes (or vertices), elements in $E(G)$ are called edges.
\end{definition}

We also provide the definition of a more specific class of graphs.
\begin{definition}[A simple graph] A simple graph is an undirected graph that contains neither loops nor multiple edges; that is, no vertex is connected to itself, and there is at most one edge between any pair of vertices.
\end{definition}

To compare two graphs, we now introduce the notion of graph homomorphisms.
\begin{definition}[Graph homomorphism]
Let $F=(V(F),E(F))$ and ${G}=(V({G}),E({G}))$ be  simple graphs.
A  graph homomorphism from $F$ to $G$ is a map
\[
\varphi: V(F)\to V({G})
\]
such that for every edge $(u,v)\in E(F)$ we have $(\varphi(u),\varphi(v))\in E({G})$.
\end{definition}

With this notion of graph homomorphisms, we can define a quantitative notion of similarity between graphs.
\begin{definition}[Homomorphism density]
The homomorphism density of $F$ in ${ G}$ is
\[
t(F,{G})\;=\;\frac{\mathrm{hom}(F,{G})}{|V({G})|^{\,|V(F)|}}\,
\]
where  $\mathrm{hom}(F,{G})$ is the number of graph homomorphisms.
\end{definition}

Note that $t(F,G)$ corresponds to the probability that a map from $V(F)$ to $V(G)$, drawn uniformly at random, is a graph homomorphism. Homomorphism densities naturally lead to a notion of convergence for sequences of graphs.

\begin{definition}[Convergence of a sequence of graphs]
The sequence of graphs $(G_N)_{N \geq 0}$ with $G_N=(V(G_N),E(G_N))$ for all $N \geq 0$ is called  convergent if for every simple graph $F$, $t(F,G_N)$ is convergent. 
\end{definition}

The limit objects of such convergent sequences are described by graphons.
\begin{proposition}
For  every convergent sequence of simple graphs, there is $w$  a graphon, i.e. a measurable symmetric function from $[0,1]^2$ to $[0,1]$ such that $$t(F,G_N) \xrightarrow[N \to \infty]{} t(F,w) :=\int_{I^{|V(F)|}} \prod_{(i,j) \in E(F)} w(\xi_i, \xi_j) d\xi_1 \dots  d\xi_N$$ for every simple graph $F$.
\end{proposition}

We now introduce the functional setting in which graphons will be considered in this paper.
\begin{definition}[Graphon space]
Let $W>0$, let $\mathcal{G}_W$ be  the weighted simple graphon space defined as
$$\mathcal{G}_W:=\{w\in L^\infty_+([0,1]^2):\,\Vert w\Vert_{L^\infty}\leq W,\,\mbox{and }w\mbox{ is symmetric}\}.$$
and $\mathcal{G}:=\mathcal{G}_1$ be the (simple) graphon space
\end{definition}
Defining $I_i^N:=((i-1)/N,i/N]$ and $\1_i^N:=\1_{I_i^N}$, every simple graph can be canonically associated to a graphon by the following map:

\begin{equation}
    G\mapsto w^G(\xi,\zeta):=\sum_{i=1}^N\sum_{j=1}^NW_{ij}\1_i^N(\xi)\1_j^N(\zeta),
\end{equation}

where $W_{ij}:=\1_{(i,j)\in E(G)}$ in the context of simple graphs and $W_{ij}$ is the weight matrix in the context of weighted graph.\\

We now introduce the cut distance, which provides a natural way to compare graphons, up to relabelling, and plays a central role in the theory of dense graph limits.

\begin{definition}[Cut distance]
For any two graphons $w,\bar w\in \mathcal{G}_W$,
\begin{center}$\begin{array}{rl}
\text{labelled cut distance:} & \displaystyle
d_\square(w,\bar w):=\sup_{S,T\subset [0,1]}\left\vert\iint_{S\times T} (w(\xi,\zeta)-\bar w(\xi,\zeta))\,d\xi\,d\zeta\right\vert, \\
\text{(unlabelled) cut distance:} & 
\delta_\square(w,\bar w)=\inf_{\Phi} d_\square(w,\bar w^\Phi),
\end{array}
$
\end{center}
with $\Phi:[0,1]\longrightarrow [0,1]$ bijective, measure-preserving  and $\bar w^\Phi(\xi,\zeta)=\bar w(\Phi(\xi),\Phi(\zeta))$.
\end{definition}

A cornerstone result of graphon theory is that convergence in homomorphism densities can be metrized via the cut distance. More precisely, convergence of a sequence of graphs is equivalent to convergence of the associated graphons in the cut distance, as stated in the following theorem.

\begin{theorem}[Lovasz-Szegedy, 2006]
$$t(F,G_N) \xrightarrow[N \to \infty]{} t(F,w) \text{ for every simple graph } F  \Longleftrightarrow  \, \delta_\square(w^{G_N},  w) \xrightarrow[N \to \infty]{} 0 .$$
\end{theorem}

An important structural property of this metric space is compactness.
\begin{proposition}\label{prop:completness}
$(\mathcal{G}_W, \delta_\square)$ is compact.
\end{proposition} 

In the remainder of the paper, we will then rely on convergence in cut distance to characterize the convergence of sequences of graphs since this metric is  naturally adapted to the analytical tools required in our mean-field approach. We also recall a useful comparison between the cut distance and the $L^1$-norm.
\begin{proposition}[Comparison with the $L^1$-norm]\label{prop:comparisonL1}
 For all $w,\bar w\in \mathcal{G}_W$,
$$d_\square(w,\bar w)\leq \Vert w-\bar w\Vert_{L^1([0,1]^2)}.$$
\end{proposition}
We emphasize that this result means that  $L^1$-topology is strictly stronger than the cut distance topology, since convergence in $L^1$ entails convergence in cut distance, while the reverse implication fails in general.\\

\begin{proposition}[Comparison with the $(L^\infty,L^1)$-norm]\label{prop:comparisonLinftyL1}
Defining 
$$
\norm{w}_{\infty\to1}:=\sup_{\norm{g}_\infty\leq 1}\int_0^1\abs{\int_0^1w(\xi,\zeta)g(\zeta)\,d\zeta}d\xi=\sup_{\norm{f}_\infty,\norm{g}_\infty\leq 1}\int_0^1\int_0^1w(\xi,\zeta)f(\xi)g(\zeta)\,d\zeta d\xi,
$$
 for all $w,\bar w\in \mathcal{G}_W$, we have
\begin{equation}\label{eq:cut_infty}
    \dfrac{1}{4}\norm{w-\Tilde w}_{\infty\to1}\leq d_\square(w, \Tilde w)\leq \norm{w-\Tilde w}_{\infty\to1}.
\end{equation}
\end{proposition}

Finally, since the interaction structures considered in this work are naturally described by weighted graphs, we introduce the following notion.
\begin{definition}[Weighted graph] A weighted graph is a triplet $G=(V(G),E(G),W(G))$ with $V(G)$ a finite set, $E(G)=V(G)^2$ and $W(G)=(W_{ij})_{1 \leq i,j \leq |V|} \in \R^{|V| \times |V(G)|}$ where we denote $|V(G)|$ the cardinal of $V(G)$. Elements in $V(G)$ are called nodes (or vertices), elements in $E(G)$ are called edges and for $1 \leq i,j \leq |V(G)|$,  $W_{ij}$ is what is called the weight attributed to the edge $(i,j)$. 
\end{definition}
All the notions and results recalled above extend without difficulty from simple graphs to weighted graphs with uniformly bounded weights.
 
\subsection{Functional spaces}\label{sec:spaces}

As is now standard in non-exchangeable frameworks, the solution to the Vlasov equation is not given by a single probability density, but rather by a parameterized family of probability measures. Such families, commonly referred to as Borel families of probability measures, capture the dependence of the distribution on the underlying parameters. For the sake of completeness, we recall the precise definition below. \\

Let us first introduce some basic notations. Let $B(0,r)$ be the open ball of radius $r$ and $I$ the interval $[0,1]$. Let $\P(\Omega)$ be the set of Borel probability measure on a borel set $\Omega$. For $\mu\in\P(\Omega)$, let $\Supp(\mu)$ be the support of $\mu$ be defined by
$$
\Supp(\mu):=\overline{\left\{z\in \Omega\ \middle|\ \forall V\in\V(z),\ \mu(V)>0\right\}}
$$
where $\V(z)$ denotes the set of measurable neighborhood of $z\in\R^n$. With this definition, we easily proves that for all measurable function $\varphi:\Omega\to\mathbb{R}$,
$$
\int_\Omega\varphi(x)\,\mu(dx)=\int_{\Supp(\mu)}\varphi(x)\,\mu(dx).
$$
If $\T:\Omega\to \Omega'$ is a measurable map and $\mu\in\P(\Omega)$, we denote $\T_{\#}\nu\in\P(\Omega')$ the pushforward measure defined for all measurable set $B$ by
$$
\T_{\#}\mu(B)=\mu(\T^{-1}(B)),
$$
or equivalently, for all measurable function $\varphi:\Omega'\to\mathbb{R}$ by
$$
\int_{\Omega'}\varphi\,d\T_{\#}\mu=\int_\Omega\varphi\circ\T\,d\mu.
$$

\begin{definition}[Borel family of probability measures]\label{defi:Borel-family-probability-measures}
Let $(\mu^\xi)_{\xi\in I}\subset \mathcal{P}(\R^{2d})$ be a parameterized family of probability measures defined for a.e. $\xi\in I$. We say that $(\mu^\xi)_{\xi\in I}$ is a Borel family if the map $\xi\in I \longmapsto \mu^\xi(B)$ is Borel-measurable for every Borel set $B\subset \R^{2d}$.
\end{definition}
This notion naturally leads to the following definition.
\begin{definition}[Fibered probability measures]\label{defi:fibered-probability-measures}
Consider any $\nu\in \mathcal{P}(I)$. We define the space of fibered probability measures by
$$\mathcal{P}_\nu (I \times \R^{2d}):=\{\mu\in \mathcal{P}(I \times \R^{2d}):\,\pi_{\xi\#}\mu=\nu\},$$
where $\pi_\xi(\xi, x, v)=\xi$ is the projection on the first component, and therefore $\pi_{\xi\#}\mu$ stands for the marginal of $\mu$ in the first component. 
\end{definition}

The classical disintegration theorem establishes a connection between these two definitions.

\begin{theorem}[Disintegration theorem]\label{theo:disintegration}
Consider any $\nu \in\mathcal{P}(I)$ and $\mu\in \mathcal{P}_\nu(I \times \R^{2d})$, then there exists a $\nu$-a.e. uniquely defined Borel family $(\mu^\xi)_{\xi\in I}\subset \mathcal{P}(\R^{2d})$ so that
$$\iint_{I \times \R^{2d}}\varphi(\xi,x,v)\,\mu(d\xi,dx,dv)=\int_0^1\left(\int_{\R^{2d}}\varphi(\xi,x,v)\,\mu^\xi(dx,dv)\right)\,\nu(d\xi),$$
for every bounded Borel-measurable map $\varphi:I \times \R^{2d} \longrightarrow \mathbb{R}$. Conversely, given any Borel family $(\mu^\xi)_{\xi\in I}$, we can associate a unique fibered probability measure $\mu\in \mathcal{P}_\nu(I \times \R^{2d})$, which we denote by $\mu^\xi\otimes \nu(d\xi)$, so that the above formula holds true.
\end{theorem}
Thus, in the rest of the paper, we will often identify measure $\mu \in \mathcal{P}_\nu( I \times \R^{2d})$ with their associated Borel families of almost everywhere defined measures $(\mu^{\xi})_{\xi \in I} \subset \mathcal{P}(\R^{2d})$. \\

Let us now introduce the topology we will adopt throughout the paper. We begin by recalling the definition of the Wasserstein space. For any two probability measures $\rho$ and $\Tilde \rho$ on $\R^{2d}$ and $p \in [1,\infty)$, the Wasserstein distance of order $p$ between $\rho$ and $\Tilde \rho$ is defined by the formula 
\begin{equation}
\W_p(\rho,\Tilde{\rho}) := \inf_{\ \gamma\in\EuScript{C}(\rho,\Tilde{\rho})} \left(\int_{\R^{4d}}\norm{(x,v)-(y,w)}^p\,\gamma(dx,dv,dy,dw)\right)^{1/p},
\end{equation}
where $\EuScript{C}(\rho,\Tilde{\rho})$ is the set of all transference plans of $\rho$ and $\Tilde \rho$ i.e.

$$
\EuScript{C}(\rho,\Tilde{\rho}) :=\{\gamma\in\P((\R^{2d})^2)\ |\ \pi_{x,v\#}\gamma=\rho,\ \pi_{y,w\#}\gamma=\Tilde{\rho}\}.
$$

We then introduce the notion of fibered Wasserstein spaces.
\begin{definition}[Fibered Wasserstein distances and spaces] For any $\nu \in \mathcal{P}(I)$, we define the fibered Wasserstein space of order $p\geq1$ as
$$\P_{p,\nu}(I\times\R^{2d}) := \left\{\mu\in\P_\nu(I \times\R^{2d}) : \W_{p,\nu}( \mu, \delta_{(0,0)}\otimes \nu) < +\infty \right\},\quad$$
with the fibered Wasserstein distance $\W_{p,\nu}$ defined as 
\begin{equation}\label{eq:def_fiberedW}
    \W_{p,\nu}(\mu,\Tilde{\mu}) := \left(\int_{I} \W_p(\mu^\xi,\Tilde \mu^\xi)^p\,\nu(d\xi)\right)^{1/p},
\end{equation}
for any $\mu, \Tilde{\mu} \in \P_{p,\nu}(I\times\R^{2d})$.
\end{definition}

\begin{remark}
We will not go into detailed about $\left(\P_{p,\nu}(I\times\R^{2d}), \W_{p,\nu}\right)$ and we refer to \cite{peszek2023heterogeneous} for a comprehensive introduction of this space. We only mention that $\W_{p,\nu}$ is not equivalent to the Wasserstein distance $\W_p$ on $I\times\R^{2d}$. In particular, if $(\mu_n)_{n\in\N}$ is a sequence of $\P_{p,\nu}(I\times\R^{2d})$ that converges to some $\mu$ for $\W_p$, then $\mu^\xi_n$ converges to $\mu^\xi$ $\nu$-a.e. does not hold in general (see \cite[Example 3.15]{peszek2023heterogeneous}). However, we always have $\W_p(\mu,\Tilde{\mu})\leq \W_{p,\nu} (\mu,\Tilde{\mu})$ (see \cite[Proposition 3.10]{peszek2023heterogeneous}).
\end{remark}

In the  rest of the paper, we will work with  the space $\P_{2,\nu}(I\times\R^{2d})$, endowed with the topology induced by the distance $\W_{2,\nu}$. Accordingly, when considering the trajectory of a time-dependent family 
$t \in [0,T] \mapsto \mu_t \in \P_{2,\nu}(I\times\R^{2d})$, continuity is always understood with respect to $\W_{2,\nu}$. In order to describe such trajectories, we introduce the space $\C([0,T],\mathcal{P}_\nu(I \times \mathbb{R}^{2d}))$ 
endowed with the metric
$$\W_{2,\nu}^\infty (\mu,\Tilde{\mu}):=\sup_{t\in[0,T]}\W_{2,\nu}(\mu_t,\Tilde{\mu}_t).$$ Given $t\in [0,T]$, we may disintegrate each $\mu_t$ using Theorem \ref{theo:disintegration-time-dependent}. However, the resulting family $(\mu_t^\xi)_{\xi\in I}$ may only be defined up to a $\nu$-null set depending on $t$, and this does not ensure the measurability in time of $t \in [0,T] \mapsto \mu_t^\xi$ for a.e. $\xi \in I$. However, as shown in \cite{AyiPoyatoPouradierDuteil} in a more complex framework, a time-dependent disintegration of $(\mu_t)_{t \in [0,T]}$ with appropriate measurability properties can be obtained. \\

\begin{theorem}[Time-dependent disintegrations]\label{theo:disintegration-time-dependent}
Consider any $\nu\in \mathcal{P}(I)$ and any Borel family of probability measures $(\mu_t)_{t\in [0,T]}\subset \P_{\nu}(I\times\R^{2d}) $. Then, there exists a Borel family $(\mu_t^\xi)_{(t,\xi)\in [0,T]\times I}\subset \mathcal{P}(\R^{2d})$ defined for $dt\otimes\nu$-a.e. $(t,\xi)\in[0,T]\times I$ such that
\begin{equation}\label{eq:disintegration-time-dependent}
\int_0^T\left(\iint_{I \times \R^{2d}}\varphi(t,\xi,x,v)\,\mu_t(d\xi,dx,dv)\right)\,dt=\iint_{[0,T]\times I}\left(\int_{\R^{2d}}\varphi(t,\xi,x,v)\,\mu_t^\xi(dx,dv)\right)\,\nu(d\xi)dt,
\end{equation}
for every bounded Borel-measurable map $\varphi:[0,T]\times I \times \R^{2d} \longrightarrow \mathbb{R}$. In particular,
\begin{enumerate}[label=(\roman*)]
\item For a.e. $t\in [0,T]$, the slice $(\mu_t^\xi)_{\xi\in I}$ is a Borel family defined for $\nu$-a.e. $\xi\in I$, and it corresponds to a possible disintegration of $\mu_t$ in the sense of Theorem \ref{theo:disintegration}.
\item For $\nu$-a.e. $\xi\in I$, the slice $(\mu_t^\xi)_{t\in [0,T]}$ is a Borel family defined for a.e. $t\in [0,T]$.
\end{enumerate}
\end{theorem}

\subsection{Main result}\label{sec:main_result}

Before stating the main result of this paper, we introduce the assumptions on the graphon $a$,  the collision avoidance function $\phi$, the communication kernel $K$ and on the measure $\nu$. We first assume that the functions $K$ and $\phi$ satisfy the following regularity hypotheses, which will remain in force for the rest of the analysis:

\begin{assumption}\label{ass:loc}
We assume that:
\begin{enumerate}
    \item $K$ and $\phi$ are Lipschitz on any compact set.\label{ass:loc_lip}
    \item For all $(x,v)\in\R^{2d},\ K(x,v)\leq \Delta_{K}(1+\norm{x}+\norm{v})$ where $\Delta_{K}>0$.\label{ass:inc_psigamma}
    \item For all $r\in\R_+,\ \abs{\phi(r)r}\leq \Delta_\phi$ \label{ass:dec_phi} where $\Delta_\phi>0$.
    \item $a$ belongs to $L^\infty(I^2)$. \label{ass:graphon}
    \item $\nu$ is absolutely continuous with respect to the Lebesgue measure on $I$. \label{ass:nu_loc}
\end{enumerate}
\end{assumption}

\begin{remark}
In the case where $K(x,v)=\psi(\norm{x})\Gamma(v)$, according to \cite{Maupoux23} and \cite{Carrillo2017}, the collision avoidance regime is reach only if $\psi$ or $\phi$ are singular at the origin. However, Assumption~\ref{ass:loc}.\ref{ass:loc_lip} implies that $K$ and $\phi$ are continuous on $\R^{2d}$ and $\R_+$ respectively. Consequently, this assumption does not meet the conditions for collision avoidance. Note also that Assumption~\ref{ass:loc}.\ref{ass:loc_lip} and Assumption~\ref{ass:loc}.\ref{ass:dec_phi} imply that
$$
\phi(r) = \dfrac{g(r)}{r}\quad\text{where}\quad\abs{g(r)}<\Delta_\phi,\ \lim_{r\to0^+}\abs{\dfrac{g(r)}{r}}<+\infty.
$$
A typical example would be for instance $\phi(r) = \min(1,1/r)$. These assumptions will be used to show the well-posedness of the characteristic equation as well as the local existence and uniqueness of a solution to the kinetic model. In the case where $\alpha=1$, we will also see that they imply that the support of the solution does not explode and therefore that it exists globally. Note that Assumption~\ref{ass:loc}.\ref{ass:nu_loc} is required, in particular, to ensure that $F[\mu](\xi,x,v)$ (see \eqref{eq:force}) is well-defined, given that $a$ is defined almost everywhere.  
\end{remark}

\begin{assumption}\label{ass:glob}
If $\alpha>1$, we assume that:
\begin{enumerate}
    \item For all $(x,v)\in\R^{2d}$, $K(x,-v)=-K(x,v)$ and $v\cdot K(x,v)\geq \Lambda_K\norm{v}^{2\alpha}$ where $\Lambda_K>0$. \label{ass:K}
    \item $a$ is almost everywhere symmetric and bounded from below, i.e $\{(\xi,\zeta)\in I^2\ |\ a(\xi,\zeta)\ne a(\zeta,\xi)\}$ is null and $\Lambda_a:=\essinf_{\xi,\zeta\in I}a(\xi,\zeta)>0$. \label{ass:a}
    \item $\Delta_\phi \leq 2^{\alpha-1/2}\Lambda_a \Lambda_K$.\label{ass:cst}
\end{enumerate}
\end{assumption}

\begin{remark} 
If $K(x,v)=\psi(\norm{x})\Gamma(v)$ then it is sufficient that $\psi\geq \Lambda_\psi$, $\Gamma(-v)=-\Gamma(v)$ and $v\cdot \Gamma(v)\geq \Lambda_\Gamma\norm{v}^{2\alpha}$ for Assumption~\ref{ass:glob}.\ref{ass:K} to be satisfied. A typical example of such a function $\Gamma$ is given by 
$$
\Gamma(v)=\norm{v}^{2\alpha-2}v.
$$
Furthermore, the fact that $v\cdot K(x,v)\geq \Lambda_K\norm{v}^{2\alpha}$ implies that $K(x_j-x_i,v_j-v_i)$ belongs to a positive cone centered at $v_j-v_i$ which induces the alignment of the velocities. 
\end{remark}

\begin{remark}
As shown in Section \ref{sec:global}, these additional Assumption~\ref{ass:glob} is required, when establishing global well-posedness and uniqueness of the Vlasov-type equation, to ensure the necessary uniform estimates and to prevent blow-up phenomena. It is not necessary when $\alpha=1$.
\end{remark}

\begin{assumption}\label{ass:integrability}
We have $K\in L^1(\R^{2d})$ and its Fourier transform $\widehat{K}\in L^1(\R^{2d})$.
\end{assumption}

\begin{remark} 
A sufficient condition for $K$ to satisfy Assumption~\ref{ass:integrability} is that $K\in W^{2d+1,1}(\mathbb{R}^{2d})$ (that is $K$ and its derivatives of order at most $2d+1$ are in $L^1(\mathbb{R}^{2d})$). Thus, this hypothesis can be interpreted as a condition on the decay and regularity of $K$.
\end{remark}

\begin{assumption}\label{ass:nu}
There exists $\Delta_\nu>0$ such that $\nu(d\xi)\leq \Delta_\nu d\xi$.
\end{assumption}

\begin{remark}
Such as many works of the existing literature, our mean-field result holds when $\nu$ is the Lebesgue measure on $I$. However, for the sake of generality, the stability result used to derive the mean-field limit is stated in a more general setting, where heterogeneous weights are assigned to the vertices. With Assumption~\ref{ass:nu}, we assume that $\nu$ admits an $L^\infty$-density with respect to the Lebesgue measure on $I$. This restriction ensures that the analysis remains essentially unchanged: arguments established in the Lebesgue case extend to this non-uniform setting without additional technical difficulties.  Finally, we note that the Lebesgue measure on $I$ satisfies Assumption~\ref{ass:nu}.
\end{remark}

We can now state our main result,  which establishes the non-exchangeable mean-field limit for the Cucker–Dong model.

\begin{theorem}[Mean-Field limit] \label{theo:mfl}
Let $a$, $K$, $\phi$ satisfy Assumptions~\ref{ass:loc} if $\alpha=1$ or Assumptions~\ref{ass:loc} and \ref{ass:glob} if $\alpha>1$ and Assumption~\ref{ass:integrability}. We assume that $\nu=\lambda$ where $\lambda$ is the Lebesgue measure on $I$. Let $\mu_0 \in \mathcal{P}_\lambda(I \times \R^{2d})$ be such that $\Supp(\mu_0)\subset B(0,r_0)$ for  some $r_0>0$.  Consider a sequence $\mu_0^N$ of measures defined as
$$
\mu_0^N := \left(\sum_{i=1}^N \delta_{(x_i(0),v_i(0))}\1_i^N(\xi)\right)\otimes \lambda(d\xi),
$$ 
such that 
$$ 
\lim_{N \to \infty} \W_{2,\lambda}(\mu_0^N, \mu_0) = 0.
$$
For $t \in [0,T]$, let $\mu_t^N$ be the empirical measure defined as
$$
\mu_t^N := \left(\sum_{i=1}^N \delta_{(x_i(t),v_i(t))}\1_i^N(\xi)\right)\otimes \lambda(d\xi),
$$ 
where $(x_i(t),v_i(t))$ is a solution to the particle system \eqref{eq:CD} with initial data $(x_i(0),v_i(0))$. Suppose that $$a_N(\xi,\zeta) = \sum_{i=1}^N\sum_{j=1}^NA_{ij}\1_i^N(\xi)\1_j^N(\zeta)$$ where $A=(A_{ij})_{1 \leq i,j \leq N}$ is the interaction matrix in \eqref{eq:CD}  satisfy the condition that there exists a subsequence $(a_{N_k})_{k \geq 0}$ such that $$\displaystyle \lim_{k \to \infty} \delta_\square(a_{N_k},a)=0.$$ This condition implies that there exist some measure-preserving maps $\Phi_k:[0,1]\longrightarrow [0,1]$  such that $$\displaystyle  \lim_{k \to \infty} d_\square(a_{N_k}^{\Phi_k},a) = 0$$ where $a_{N_k}^{\Phi_k}(\xi,\zeta) := a_{N_k}^{\Phi_k}(\Phi_k(\xi),\Phi_k(\zeta))$. Then, we have  $$ \lim_{N \to \infty} \W_{2,\nu}^\infty(\mu^{N_k,\Phi_k}, \mu) = 0$$  where 
 for every $\xi \in [0,1]$, $$\mu^{N_k,\Phi_k,\xi} :=\mu^{N_k,\Phi_k(\xi)} $$ and $\mu \in \C(\R_+,\mathcal{P}_{2,\nu}(I \times \R^{2d}))$ is the unique solution to the kinetic Cucker-Dong equation  \eqref{eq:KCD1} with force \eqref{eq:force} and continuous alignment measure \eqref{eq:alignment_measure_cont}.
 
\end{theorem}

\begin{remark}
We formulate our result under the general assumption that the sequence of graphs satisfy  the condition that there exists a subsequence $(a_{N_k})_{k \geq 0}$ such that $$\displaystyle \lim_{k \to \infty} \delta_\square(a_{N_k},a)=0.$$
We now describe two typical situations in which this condition holds. \\

\noindent\textbf{Case 1: Uniformly bounded interaction matrices.}

Assume that the sequence $(A^{N})_{N\ge1}$ is uniformly bounded in $N$, that is,
\[
\sup_{N \geq 0}\|A^{N}\|_\infty < \infty.
\]
Then the associated sequence $(a_N)_{N \geq 0}$ is uniformly bounded in $L^\infty$, and therefore belongs to a fixed graphon space $\mathcal{G}_W$. 
Since $(\mathcal{G}_W,\delta_\square)$ is compact, there exists a subsequence  converging in cut distance to a limit graphon $a$. 
This limit function $a$ is precisely the one appearing in the  Vlasov-type equation obtained in the mean-field limit.

\medskip

\noindent\textbf{Case 2: Approximation of a prescribed graphon.}

Let $a \in \mathcal{G}_W$ be given and define the interaction matrices by
\[
A_{ij}^N = N^2\iint_{I_i^N\times I_j^N}a(\xi,\zeta)\,d\zeta\,d\xi.
\]
A direct application of the Lebesgue differentiation theorem together with the dominated convergence theorem yields
\[
\lim_{N\to\infty}\|a_N - a\|_{L^1(I^2)}=0.
\]
By Proposition \ref{prop:comparisonL1}, this yields convergence in the (unlabelled) cut distance. 
In particular, no extraction of a subsequence is required, and the convergence occurs with the trivial relabelling $\Phi_k = \mathrm{Id}$.

\end{remark}

\section{Well-posedness of the Kinetic Cucker-Dong model}\label{sec:well_posedness}

We begin this section by rewriting the Kinetic Cucker–Dong equation in a form that is compatible with the notion of fibered probability measures  $\mu_t \in \mathcal{P}_{2, \nu} (I \times \R^{2d})$ introduced in Section \ref{sec:spaces}. In this formulation, the equation takes the form

\begin{equation}\label{eq:KCD}
\partial_t \mu_t+v \cdot \nabla_x \mu_t+ \nabla_v \cdot  \left(F_a[\mu_t]\,\mu_t\right)=0, \quad t \in [0,T],\,x \in \mathbb{R}^d,\,v \in \mathbb{R}^d,
\end{equation}

where  $F_a[\mu_t]$ is defined in equation \eqref{eq:force}. The goal of this section is to prove the existence and uniqueness of the solution in $\C(\R_+,\mathcal{P}_{2,\nu}(I \times \R^{2d}))$  as stated in the following result.

\begin{theorem}[Global well-posedness of the Vlasov equation]\label{theo:well-posedness-vlasov-global}
Let $a$, $K$, $\phi$ and $\nu$ satisfy Assumptions~\ref{ass:loc} if $\alpha=1$ or Assumptions~\ref{ass:loc} and \ref{ass:glob} if $\alpha>1$. Let the initial datum $\mu_0 \in \mathcal{P}_{2, \nu} (I \times \R^{2d})$ be such that, for a.e. $\xi \in I$, $\Supp(\mu_0^\xi)\subset B(0,r_0)$, for some $r_0>0$. Then, there exists  a unique solution $\mu\in C(\R_+,\mathcal{P}_{2, \nu} (I \times \R^{2d}))$  to \eqref{eq:KCD} with initial condition ${\mu_t\big|}_{t=0}=\mu_0$. Moreover, the solution satisfies the  following property: for all $T>0$, there exists $r(T)>0$  such that for all  $t\in[0,T]$ and for a.e. $\xi\in I$, $\Supp(\mu_t^\xi)\subset B(0,r(T))$.
\end{theorem}

\subsection{Well-posedness of the characteristics}

We shall solve the Vlasov equation by applying the method of characteristics. This reduces the problem to the study of the following associated system of ordinary differential equations, from which the solution to the original equation can be reconstructed:

\begin{equation}\label{eq:charac}
\left\{\begin{array}{l}
\begin{array}{rcl}
\dsp \frac{d}{dt} X_a [\mu](t,{s},\xi,x,v) & = & \dsp V_a [\mu](t,{s},\xi,x,v),  \\
\dsp \frac{d}{dt} V_a [\mu](t,{s},\xi,x,v)   & = & F_a[\mu_t](\xi, X_a [\mu](t,{s},\xi,x,v), V_a[\mu](t,{s},\xi,x,v)), 
\end{array} \\
\dsp ~\, X_a [\mu]({s},{s},\xi,x,v) = x,\quad V_a[\mu]({s},{s},\xi,x,v) =v, 
\end{array} \right.
\end{equation}

Let us prove the local existence and uniqueness of the solutions of \eqref{eq:charac}.

\begin{lemma}\label{lem:estim_sigma}
For all $\mu,\Tilde{\mu}\in\P_{2,\nu}(I\times\R^{2d})$, $\sigma_v(\mu),\sigma_v(\Tilde{\mu})\in\R_+$ and we have
$$
\abs{\sigma_v(\mu)-\sigma_v(\Tilde{\mu})}\leq\left(1+\sqrt{d}\right)\W_{2,\nu}(\mu,\Tilde{\mu}).
$$
\end{lemma}

\begin{proof}
First, let us note that for all $\mu\in\P_{2,\nu}(I\times\R^{2d})$, we have
$$
\sigma_v(\mu)^2=\iint_{I\times\R^{2d}}\norm{w-m_v(\mu)}^2\,\mu^\xi(dy,dw)\nu(d\xi)\quad\text{where}\quad m_v(\mu)=\iint_{I\times\R^{2d}}w\,\mu^\xi(dy,dw)\nu(d\xi),
$$
and thus $\sigma_v(\mu)=\W_2(\pi_{v\#}\mu,\delta_{m_v(\mu)})$ where $\pi_v(\xi, x, v)=v$ is the projection on the third component, which leads to $\sigma_v(\mu)\in\R_+$. Furthermore, we also have
$$
\begin{aligned}
\abs{\sigma_v(\mu)-\sigma_v(\Tilde{\mu})} &= \abs{\W_2(\pi_{v\#}\mu,\delta_{m_v(\mu)})-\W_2(\pi_{v\#}\Tilde{\mu},\delta_{m_v(\Tilde{\mu})})} \\
&\leq\abs{\W_2(\pi_{v\#}\mu,\delta_{m_v(\mu)})-\W_2(\pi_{v\#}\Tilde{\mu},\delta_{m_v(\mu)})}+\abs{\W_2(\pi_{v\#}\Tilde{\mu},\delta_{m_v(\mu)})-\W_2(\pi_{v\#}\Tilde{\mu},\delta_{m_v(\Tilde{\mu})})} \\
&\leq\W_2(\pi_{v\#}\mu,\pi_{v\#}\Tilde{\mu})+\norm{m_v(\mu)-m_v(\Tilde{\mu})}.
\end{aligned}
$$
In the one hand, from \cite[Equation (3.13)]{peszek2023heterogeneous}, we have
$$
\W_2(\pi_{v\#}\mu,\pi_{v\#}\Tilde{\mu})\leq \W_2(\mu,\Tilde{\mu})\leq\W_{2,\nu}(\mu,\Tilde{\mu}).
$$
On the other hand, from the dual characterisation of $\W_1$ and the fact that $\W_1\leq\W_2$, we have
$$
\norm{m_v(\mu)-m_v(\Tilde{\mu})}\leq\sqrt{d}\,\W_1(\pi_{v\#}\mu,\pi_{v\#}\Tilde{\mu})\leq\sqrt{d}\,\W_2(\pi_{v\#}\mu,\pi_{v\#}\Tilde{\mu})\leq\sqrt{d}\,\W_{2,\nu}(\mu,\Tilde{\mu}),
$$
which concludes the proof.
\end{proof}

\begin{lemma}\label{lem:estim_F_mu}
Let $a$, $K$, $\phi$ and $\nu$ satisfy Assumption~\ref{ass:loc}. Consider two probability measures $\mu,\Tilde{\mu}\in \mathcal{P}_\nu(I \times \R^{2d})$ such that for a.e $\xi\in I$, $\text{supp}(\mu^\xi),\text{supp}(\Tilde{\mu}^\xi) \subset B(0,r)$ for some $r>0$. Then we have for all $(\xi,x,v)\in I\times B(0,r)$,
$$
\norm{F_a[\mu](\xi,x,v) - F_a[\Tilde{\mu}](\xi,x,v)} \leq L_{1,r}\W_{2,\nu}(\mu,\Tilde{\mu}),
$$
where $L_{1,r}$ only depends on $d$, $r$, $a$, $\alpha$, $K$ and $\phi$.
\end{lemma}

\begin{proof}
First we have

$$
\begin{aligned}
    \norm{F_a[\mu](\xi,x,v) - F_a[\Tilde{\mu}](\xi,x,v)} &\leq \norm{\iint_{I\times B(0,r)}a(\xi,\zeta)K(y-x,w-v)\,(\mu^\zeta-\Tilde{\mu}^\zeta)(dy,dw)\nu(d\zeta)} \\
    &+\abs{\sigma_v(\mu)^{2\alpha - 1}-\sigma_v(\Tilde{\mu})^{2\alpha - 1}}\iint_{I\times B(0,r)}\phi(\norm{x-y})\norm{x-y}\,\mu^\zeta(dy,dw)\nu(d\zeta) \\
    &+\sigma_v(\Tilde\mu)^{2\alpha - 1}\norm{\iint_{I\times B(0,r)}\phi(\norm{x-y})(x-y)\,(\mu^\zeta-\Tilde{\mu}^\zeta)(dy,dw)\nu(d\zeta)}.
\end{aligned}
$$

From Assumption~\ref{ass:loc}.\ref{ass:loc_lip}, the fact that the product of two Lipschitz on every compact set functions is Lipschitz on every compact set, $a\in L^\infty\left(I^2\right)$, the dual characterisation of $\W_1$ \cite[Equation (6.3)]{villani2008optimal} and the fact that $\W_1\leq\W_2$, we have

$$
\norm{\iint_{I\times B(0,r)}a(\xi,\zeta)K(y-x,w-v)\,(\mu^\xi-\Tilde{\mu}^\xi)(dy,dw)\nu(d\xi)}\leq \sqrt{d}\norm{a}_\infty L_{K,r}\W_{2,\nu}(\mu,\Tilde{\mu}),
$$

where $L_{K,r}$ is the Lipschitz constant of $K$ on $B(0,2r)$. Then, since $\Supp(\Tilde\mu_t)\subset B(0,r)$, we have

$$
\sigma_v(\Tilde\mu)\leq \sqrt{\iint_{I\times B(0,r)}\norm{w}^2\,\mu^\zeta(dy,dw)\nu(d\zeta)}\leq r,
$$

hence, 

$$
\sigma_v(\Tilde\mu)^{2\alpha - 1}\norm{\iint_{I\times B(0,r)}\phi(\norm{x-y})(x-y)\,(\mu^\zeta-\Tilde{\mu}^\zeta)(dy,dw)\nu(d\zeta)}\leq \sqrt{d}r^{2\alpha-1}L_{\phi x,r} \W_{2,\nu}(\mu,\Tilde{\mu}),
$$

where $L_{\phi x,r}$ is the Lipschitz constant of $(x,v)\mapsto \phi(\norm{x})x$ on $B(0,2r)$. Finally, from Assumption~\ref{ass:loc}.\ref{ass:dec_phi}, we have 

$$
\iint_{I\times B(0,r)}\phi(\norm{x-y})\norm{x-y}\,\mu^\zeta(dy,dw)\nu(d\zeta) \leq \Delta_\phi,
$$

and if $\alpha\geq1$, from Lemma~\ref{lem:estim_sigma}, we have

$$
\begin{aligned}
\abs{\sigma_v(\mu)^{2\alpha - 1}-\sigma_v(\Tilde{\mu})^{2\alpha - 1}} &\leq(2\alpha-1)r^{2\alpha-2}\abs{\sigma_v(\mu)-\sigma_v(\Tilde{\mu})} \\
&\leq(2\alpha-1)r^{2\alpha-2}\left(1+\sqrt{d}\right)\W_{2,\nu}(\mu,\Tilde{\mu}),
\end{aligned}
$$

which concludes the proof.
\end{proof}

\begin{lemma}\label{lem:lipscont_L}
Let $a$, $K$, $\phi$ and $\nu$ satisfy Assumption~\ref{ass:loc}. Consider any curve of probability measures $\mu\in\C([0,T],\mathcal{P}_\nu(I \times \R^{2d}))$ such that for all $t \in [0,T]$ and for a.e $\xi\in I$, $\text{supp}(\mu_t^\xi) \subset B(0,r)$ for some $r>0$. Then $t\mapsto F_a[\mu_t](\xi,x,v)$ is continuous and $(x,v)\mapsto F_a[\mu_t](\xi,x,v)$ is Lipschitz on every compact set $\Omega\subset\R^{2d}$ with a Lipschitz constant $L_\Omega$ which only depends on $\Omega$, $r$, $\alpha$, $a$, $K$ and $\phi$. 
\end{lemma}

\begin{proof}
The continuity of $t\mapsto F_a[\mu_t](\xi,x,v)$ comes from Lemma~\ref{lem:estim_F_mu} and the continuity of $t\mapsto\mu_t$ for $\W_{2,\nu}$. Let us prove that $(x,v)\mapsto F_a[\mu_t](\xi,x,v)$ is Lipschitz on every compact set. Let $\Omega\subset \R^{2d}$ be a compact set. Then for all $(x,v),(x',v')\in \Omega$, 

$$
\begin{aligned}
    \norm{F_a[\mu_t](\xi,x,v) - F_a[\mu_t](\xi,x',v')} &\leq \iint_{I\times B(0,r)}a(\xi,\zeta)\norm{K(y-x,w-v)-K(y-x',w-v')}\,\mu_t^\zeta(dy,dw)\nu(d\zeta) \\
    &+\sigma_v(\mu_t)^{2\alpha - 1}\iint_{I\times B(0,r)}\norm{\phi(\norm{x-y})(x-y)-\phi(\norm{x'-y})(x'-y)}\,\mu_t^\zeta(dy,dw)\nu(d\zeta),
\end{aligned}
$$

which directly leads to

$$
\norm{F_a[\mu_t](\xi,x,v) - F_a[\mu_t](\xi,x',v')} \leq (\norm{a}_\infty L_{K,\Omega_r}+r^{2\alpha-1}L_{\phi x,\Omega_r})\norm{(x,v)-(x',v')}.
$$

where $L_{K,\Omega_r}$ and $L_{\phi x,\Omega_r}$ are respectively the Lipschitz constants of $K$ and $(x,v)\mapsto \phi(\norm{x})x$ on $\Omega_r$ where

$$
\Omega_r:=\left\{(x_1,v_1)-(x_2,v_2)\ \middle|\ (x_1,v_1)\in B(0,r),(x_2,v_2)\in \Omega\right\}.
$$
\end{proof}

\begin{proposition}\label{pro:T_well_defined}
Let $a$, $K$, $\phi$ and $\nu$ satisfy Assumption~\ref{ass:loc}. Consider any curve of probability measures $\mu\in\C([0,T],\mathcal{P}_\nu(I \times \R^{2d}))$ such that for all $t \in [0,T]$ and for a.e $\xi\in I$, $\text{supp}(\mu_t^\xi) \subset B(0,r)$ for some $r>0$. Then, there exists a unique maximal solution of \eqref{eq:charac} defined on an interval $J\subset[0,T]$ such that $s\in J$. 
\end{proposition}

\begin{proof}
The existence of a unique maximal solution on $J$ is a direct consequence of Lemma~\ref{lem:lipscont_L}, the fact that being Lipschitz on every compact implies being locally Lipschitz and the Picard–Lindelöf theorem.
\end{proof}

From now on, we will denote for all $s\in[0,T]$ and $t\in J$
\begin{equation}\label{eq:characteristics-flow-map}
\mathcal{T}_{{s},t}^{\xi}[a,\mu](x,v):=\left(X_a [\mu](t,{s},\xi,x,v), V_a [\mu](t,{s},\xi,x,v)\right), \quad (x,v) \in \R^{2d},\,\mbox{a.e. }\xi\in I,
\end{equation}
and for all $(\xi,x,v) \in I\times \R^{2d}$,
$$
H_a[\mu](\xi,x,v)=(v,F_a[\mu](\xi,x,v)).
$$
Note that from Lemma~\ref{lem:lipscont_L}, $(x,v)\mapsto H_a[\mu_t](\xi,x,v)$ is Lipschitz on every compact set $\Omega\subset\R^{2d}$ with a Lipschitz constant $L_{2,\Omega}=\sqrt{1+L_\Omega^2}$ which only depends on $\Omega$, $r$, $\alpha$, $a$, $K$ and $\phi$. \\

The following Lemma proves that the local solution does not explode in finite time and thus ensures that $\mathcal{T}_{s,t}^{\xi}[a,\mu](x,y)$ is actually well define for all $s,t\in[0,T]$. 

\begin{lemma}\label{lem:estim_T_t}
Let $a$, $K$, $\phi$ and $\nu$ satisfy Assumption~\ref{ass:loc}. Consider any curve of probability measures   $\mu\in\C([0,T],\mathcal{P}_\nu(I \times \R^{2d}))$ such that for all $t \in [0,T]$ and for a.e $\xi\in I$, $\text{supp}(\mu_t^\xi) \subset B(0,r)$ for some $r>0$. Then for all $t\in J$,
$$
\norm{\mathcal{T}_{s,t}^{\xi}[a,\mu](x,y)-(x,y)}\leq \dfrac{C_2}{C_1}\left(e^{C_1\abs{t-s}}-1\right),
$$
where $C_1:=1+\norm{a}_\infty \Delta_{K}$ and $C_2= C_2(x,v) :=\norm{v}+\norm{a}_\infty \Delta_{K}(1+2r+\norm{x}+\norm{v})+\Delta_\phi r^{2\alpha-1}$. 
\end{lemma}

\begin{proof}
For the sake of presentation, let us denote $X(t)=X_a [\mu](t,s,\xi,x,v)$ and $V(t)=V_a [\mu](t,s,\xi,x,v)$. Let us define $R(t)=\norm{X(t)-x}+\norm{V(t)-v}$. We have
$$
\frac{d}{dt}R(t)\leq\norm{V(t)}+\norm{F_a[\mu_t]\left(\xi,X(t),V(t)\right)}.
$$
Then, from Assumptions~\ref{ass:loc}.\ref{ass:inc_psigamma} and \ref{ass:loc}.\ref{ass:dec_phi}, we have for all $(x,v)\in\R^{2d}$,
$$
\begin{aligned}
\norm{F_a[\mu_t]\left(\xi,x,v\right)}&\leq \norm{a}_\infty \Delta_{K}\iint_{I\times B(0,r)}1+\norm{x-y}+\norm{v-w}\,\mu_t^\zeta(dy,dw)\nu(d\zeta) + \Delta_\phi r^{2\alpha-1} \\
&\leq \norm{a}_\infty \Delta_{K}\left(1+2r+\norm{x}+\norm{v}\right) + \Delta_\phi r^{2\alpha-1},
\end{aligned}
$$
which leads to
$$
\frac{d}{dt}R(t)\leq C_1R(t)+C_2.
$$
From the Grönwall's lemma, it follows that for all $t\geq s$,
$$
\norm{X(t)-x}+\norm{V(t)-v}\leq\dfrac{C_2}{C_1}\left(e^{C_1(t-s)}-1\right),
$$
Reasoning as before with $Y(t)=X(s-t)$ and $W(t)=V(s-t)$, we also prove that for all $t\leq s$,
$$
\norm{X(t)-x}+\norm{V(t)-v}\leq\dfrac{C_2}{C_1}\left(e^{C_1(s-t)}-1\right),
$$
which concludes the proof using $\norm{(x,v)}\leq\norm{x}+\norm{v}$.
\end{proof}

\subsection{Distributional solutions}
 
In this section, we specify the class of solutions under consideration by introducing the notion of distributional solution.
\begin{definition}[Distributional solutions of \eqref{eq:KCD}]\label{defi:distributional-solution-Vlasov-joint}
Let $a$ belong to $L^\infty(I^2)$ and $\mu_0$ belong to $\mathcal{P}_\nu(I \times \R^{2d})$. We say that $\mu\in C([0,T],\mathcal{P}_\nu(I \times \R^{2d}))$ is a distributional solution of the Kinetic Cucker-Dong equation \eqref{eq:KCD} with initial condition $\mu_0$ if
\begin{align}
&\int_0^T\iint_{I \times \R^{2d}}\left(\partial_t\varphi(t,\xi,x,v)+v \cdot\nabla_x\varphi(t,\xi,x,v) +  F_a[\mu_t](\xi,x,v)\cdot\nabla_v\varphi(t,\xi,x,v)\right)\,\mu_t(d\xi,dx,dv)\,dt\nonumber\\
&\qquad=-\iint_{I \times \R^{2d}}\varphi(0,\xi,x,v)\,\mu_0(d\xi,dx,dv),\label{eq:distributional-solution-joint}
\end{align}
for all $\varphi\in C^1_c([0,T)\times I \times \R^{2d})$.
\end{definition}

As explained in  \cite[Proposition 4.5]{AyiPoyatoPouradierDuteil}, standard arguments show that proving  the existence and the uniqueness of the solutions of \eqref{eq:KCD} is equivalent to solving a fixed point equation. More precisely, we have the following proposition.
\begin{proposition}[Distributional solutions of \eqref{eq:KCD}]\label{pro:distributional-solution-Vlasov-fibered}
Let $a$ belong to $L^\infty(I^2)$ and $\mu_0$ belong to $\mathcal{P}_\nu(I \times \R^{2d})$. Consider any curve of probability measures $\mu\in C([0,T],\mathcal{P}_\nu(I \times \R^{2d}))$, and its associated Borel family $(\mu_t^\xi)_{(t,\xi)\in [0,T]\times I}$ as in the disintegration Theorem \ref{theo:disintegration-time-dependent}. Then, the following conditions are equivalent:
\begin{enumerate}[label=(\roman*)]
\item $\mu$ is a distributional solution to \eqref{eq:KCD} with initial condition $\mu_0$.
\item For a.e. $\xi\in I$, $(\mu_t^\xi)_{t\in [0,T]}$ is a distributional solution to \eqref{eq:KCD1}, {\it i.e.},
\begin{multline}\label{eq:distributional-solution-fibered}
\int_0^T\int_{\R^{2d}}\left(\partial_t\varphi(t,x,v)+v \cdot\nabla_x\varphi(t,x,v) + F_a[\mu_t](\xi,x,v)\cdot\nabla_v\varphi(t,x,v)\right)\,\mu_t^\xi(dx,dv)\,dt \\ =-\int_{\mathbb{R}^d}\varphi(0,x,v)\,\mu_0^\xi(dx,dv),
\end{multline}
for all $\varphi\in C^1_c([0,T)\times \R^{2d})$.
\item For a.e. $\xi\in I$, $(\mu_t^\xi)_{t\in [0,T]}$ is the push forward along the flow map, {\it i.e.},
\begin{equation}\label{eq:distributional-solution-push-forward}
\mu_t^\xi=\mathcal{T}_{0,t}^\xi[a,\mu]_{\#}\mu_0^\xi,\quad \mbox{a.e. }t\geq 0,
\end{equation}
where $\mathcal{T}_t^\xi$ is given in \eqref{eq:characteristics-flow-map}.
\end{enumerate}
\end{proposition}

\begin{remark}
The equivalence between (i) and (ii) in Proposition~\ref{pro:distributional-solution-Vlasov-fibered} comes from the fact that the distribution of labels in our population is constant over time (choose a test function that depends only on $\xi$ in \eqref{eq:distributional-solution-joint}). In the case of more complex models including mutations or births and deaths, this property is no longer true and thus (i) and (ii) are no longer equivalent. 
\end{remark}

\subsection{Well-posedness of the Kinetic Cucker-Dong equation}

This section is dedicated to the proof of the global existence and uniqueness of \eqref{eq:KCD} under Assumptions~\ref{ass:loc} and \ref{ass:glob}. We prove the local well-posedness in Section~\ref{sec:local} and the global well-posedness in Section~\ref{sec:global}.

\subsubsection{Local existence and uniqueness}\label{sec:local}

\begin{theorem}[Local well-posedness of the Vlasov equation]\label{theo:well-posedness-vlasov}
Let $a$, $K$, $\phi$ and $\nu$ satisfy Assumption~\ref{ass:loc}. Let the initial datum $\mu_0 \in \mathcal{P}_{\nu}(I \times \R^{2d})$ be such that, for a.e. $\xi \in I$, $\Supp(\mu_0^\xi)\subset B(0,r_0)$, for some $r_0>0$. Then, for every $r > r_0$, there exists $T > 0$ such that equation \eqref{eq:KCD} admits a unique solution $\mu\in C([0,T],\mathcal{P}_{\nu}(I \times \R^{2d}))$ with initial condition ${\mu_t\big|}_{t=0}=\mu_0$ such that for a.e. $\xi \in I$, $\Supp(\mu_t^\xi)\subset B(0,r)$.
\end{theorem}

\begin{proof}
Let $\mu_0\in\P_\nu(I\times\R^{2d})$ such that for a.e. $\xi \in I$, $\Supp(\mu_0^\xi)\subset B(0,r_0)$ for some $r_0>0$. Let $\M$ be defined by
$$
\M=\{\mu\in\C([0,T],\mathcal{P}_{\nu}(I \times \R^{2d}))\ |\ {\mu_t\big|}_{t=0}=\mu_0\quad\text{and for a.e $\xi\in I$,}\ \Supp(\mu_t^\xi)\subset\mathcal{B}(0,r)\}.
$$
Using the completeness of $(\P_{2,\nu}(I\times\R^{2d}),\W_{2,\nu})$ (see \cite[Proposition 3.4]{peszek2023heterogeneous}) and classical arguments, we prove that $(\M,\W_{2,\nu}^\infty)$ is a complete metric space. For all $\mu\in\C([0,T],\mathcal{P}_\nu(I \times \R^{2d}))$, let $\F(\mu)=(\F(\mu)_t^\xi)_{t\in[0,T], \xi\in I}$ where 
$$
\F(\mu)_t^\xi = \T_{0,t}^\xi[a,\mu]_{\#}\mu_0^\xi.
$$

Let us prove that there exists $T>0$ such that $\F:\M\to\M$. First, we note that from Theorem~\ref{theo:disintegration}, we have $\F(\mu)_t\in\mathcal{P}_\nu(I \times \R^{2d})$. Since for all $(t,x,v)\in[0,T]\times B(0,r)$, we have 
$$
\norm{H_a[\mu_t](\xi,x,v)}\leq C_r := r+\norm{a}_\infty \Delta_{K}(1+4r)+\Delta_\phi r^{2\alpha-1},
$$
we have for all $(x,v)\in B(0,r_0)$ and for all $t<(r-r_0)/C_r$, $\T^\xi_{0,t}[a,\mu](x,v)\in B(0,r)$ and thus for a.e $\xi\in I$, $\Supp(\F(\mu)^\xi_t)\subset B(0,r)$. The continuity of $t\mapsto\F(\mu)_{t}$ directly comes from \cite[Lemma 3.11]{canizo2011well} and Lemma~\ref{lem:estim_T_t}, which finally proves that $\F:\M\to\M$ if $T\leq(r-r_0)/C_r$. \\

Let us proves that there exists $T\leq(r-r_0)/C_r$ such that $\F$ is a contraction on $\M$. From \cite[Lemma 3.11]{canizo2011well}, we have
$$
\W_{2,\nu}^\infty (\F(\mu),\F(\Tilde{\mu}))\leq \sup_{t\in[0,T]}\int_I\esssup_{(x,v)\in B(0,r_0)}\norm{\T^\xi_{0,t}[a,\mu](x,v)-\T^\xi_{0,t}[a,\Tilde{\mu}](x,v)}\,\nu(d\xi).
$$
Then, using the fact that
$$
\T^\xi_{0,t}[a,\mu](x,v)=(x,v)+\int_0^tH_a[\mu_s](\xi,\T^\xi_{0,s}[a,\mu](x,v))\,ds,
$$
it follows that
$$
\norm{\T^\xi_{0,t}[a,\mu](x,v)-\T^\xi_{0,t}[a,\Tilde{\mu}](x,v)}\leq I_1 + I_2,
$$
where
$$
I_1=\int_0^t\norm{H_a[\mu_s](\xi,\T^\xi_{0,s}[a,\mu](x,v))-H_a[\mu_s](\xi,\T^\xi_{0,s}[a,\Tilde{\mu}](x,v)}\,ds,
$$
and
$$
I_2=\int_0^t\norm{F_a[\mu_s](\xi,\T^\xi_{0,s}[a,\Tilde{\mu}](x,v))-F_a[\Tilde{\mu}_s](\xi,\T^\xi_{0,s}[a,\Tilde{\mu}](x,v)}\,ds.
$$
On the one hand, since $T\leq(r-r_0)/C_r$, we have $\T^\xi_{0,s}[a,\mu](x,v),\T^\xi_{0,s}[a,\Tilde{\mu}](x,v)\in B(0,r)$ and thus, from Lemma~\ref{lem:lipscont_L}, we know that there exist a positive constant that we will denote $L_{2,r}$ such that we have 
$$
I_1\leq L_{2,r}\int_0^t\norm{\T^\xi_{0,t}[a,\mu](x,v)-\T^\xi_{0,t}[a,\Tilde{\mu}](x,v)}\,ds,
$$
On the other hand, from Lemma~\ref{lem:estim_F_mu}, we have
$$
I_2\leq L_{1,r}\int_0^t\W_{2,\nu}(\mu_s,\Tilde{\mu}_s)\,ds,
$$
where $L_{1,r}$ only depends on $d$, $r$, $a$, $\alpha$, $K$ and $\phi$. Finally, it follows from the Grönwall's Lemma that for all $z_0\in\R^{2d}$,
$$
\norm{\T^\xi_{0,t}[a,\mu](x,v)-\T^\xi_{0,t}[a,\Tilde{\mu}](x,v)}\leq L_{1,r}\dfrac{e^{L_{2,r}t}-1}{L_{2,r}} \W_{2,\nu}^\infty  (\mu,\Tilde{\mu}),
$$
and thus that 
$$
\W_{2,\nu}^\infty  (\F(\mu),\F(\Tilde{\mu}))\leq L_{1,r}\dfrac{e^{L_{2,r}T}-1}{L_{2,r}} \W_{2,\nu}^\infty (\mu,\Tilde{\mu}).
$$
If $T<\dfrac{1}{L_{2,r}}\ln(1+\dfrac{L_{2,r}}{L_{1,r}})$ then $\F$ is a contraction which concludes the proof of the theorem, using the Banach fixed-point theorem. 
\end{proof}

\subsubsection{Global existence and uniqueness}\label{sec:global}

The proof of global existence relies on a classical  extension argument. Nevertheless, to carry out this argument rigorously within our framework, we need Assumption \ref{ass:glob}.  \\

\begin{lemma}\label{lem:dec_sigma}
Let $a$, $K$, $\phi$ and $\nu$ satisfy Assumptions~\ref{ass:loc} and \ref{ass:glob}. Let $\mu\in C([0,T],\mathcal{P}_{\nu}(I \times \R^{2d}))$ be a local solution of \eqref{eq:KCD} such that a.e. $\xi \in I$, $\Supp(\mu_0^\xi)\subset B(0,r_0)$ and $\Supp(\mu_t^\xi)\subset B(0,r)$. Then, we have 
$$
\dfrac{d}{dt} \sigma_v(\mu_t)^2 \leq -(2^\alpha \Lambda_a \Lambda_K-\sqrt{2}\Delta_\phi)\sigma_v(\mu_t)^{2\alpha}.
$$
In particular $\sigma_v(\mu_t)\leq\sigma_v(\mu_0)\leq r_0$.
\end{lemma}

\begin{proof}
First, note that from Assumption~\ref{ass:glob}.\ref{ass:K} and \ref{ass:glob}.\ref{ass:a}, we have
$$
\begin{aligned}
&\iint_{I^2\times\R^{4d}} \left[ a(\xi,\zeta)K(y-x,w-v)+\sigma_v(\mu_t)^{2\alpha-1}\phi(\norm{x-y})(x-y)\,\right] \mu^\zeta_t(dy,dw) \mu^\xi_t(dx,dv) \nu(d\zeta) \nu(d\xi) \\
&=\iint_{I^2\times\R^{4d}} \left[  a(\xi,\zeta)K(x-y,v-w)+\sigma_v(\mu_t)^{2\alpha-1}\phi(\norm{y-x})(y-x)\,\right] \mu^\zeta_t(dy,dw) \mu^\xi_t(dx,dv) \nu(d\zeta) \nu(d\xi) \\
&=-\iint_{I^2\times\R^{4d}} \left[ a(\xi,\zeta)K(y-x,w-v)+\sigma_v(\mu_t)^{2\alpha-1}\phi(\norm{y-x})(x-y) \right] \, \mu^\zeta_t(dy,dw) \mu^\xi_t(dx,dv) \nu(d\zeta) \nu(d\xi),
\end{aligned}
$$
and thus 
$$
\dfrac{d}{dt}\iint_{I\times\R^{2d}}v\,\mu^\xi_t(dx,dv)\nu(d\xi)=0.
$$
It follows that 
$$
\dfrac{d}{dt} \sigma_v(\mu_t)^2 =\dfrac{d}{dt}\iint_{I\times\R^{2d}}\norm{v}^2\,\mu^\xi_t(dx,dv)\nu(d\xi).
$$
Then, using Assumptions~\ref{ass:glob}.\ref{ass:K} and \ref{ass:glob}.\ref{ass:a}, we have
$$
\begin{aligned}
\dfrac{d}{dt} \sigma_v(\mu_t)^2 =&2\iint_{I^2\times\R^{4d}}a(\xi,\zeta)K(y-x,w-v)\cdot v\, \mu^\zeta_t(dy,dw) \mu^\xi_t(dx,dv) \nu(d\zeta) \nu(d\xi) \\
&+2\sigma_v(\mu_t)^{2\alpha-1}\iint_{I^2\times\R^{4d}}\phi(\norm{x-y})(x-y)\cdot v\, \mu^\zeta_t(dy,dw) \mu^\xi_t(dx,dv) \nu(d\zeta) \nu(d\xi) \\
=&-\iint_{I^2\times\R^{4d}}a(\xi,\zeta)K(y-x,w-v)\cdot (w-v)\, \mu^\zeta_t(dy,dw) \mu^\xi_t(dx,dv) \nu(d\zeta) \nu(d\xi) \\
&+\sigma_v(\mu_t)^{2\alpha-1}\iint_{I^2\times\R^{4d}}\phi(\norm{x-y})(x-y)\cdot (v-w)\, \mu^\zeta_t(dy,dw) \mu^\xi_t(dx,dv) \nu(d\zeta) \nu(d\xi) \\
\leq&-\Lambda_a\Lambda_K\iint_{I^2\times\R^{4d}}\norm{v-w}^{2\alpha}\, \mu^\zeta_t(dy,dw) \mu^\xi_t(dx,dv) \nu(d\zeta) \nu(d\xi)\\
&+\Delta_\phi\sigma_v(\mu_t)^{2\alpha-1}\iint_{I^2\times\R^{4d}}\norm{v-w}\, \mu^\zeta_t(dy,dw) \mu^\xi_t(dx,dv) \nu(d\zeta) \nu(d\xi) \\
\leq&-\Lambda_a\Lambda_K\left(\iint_{I^2\times\R^{4d}}\norm{v-w}^2\, \mu^\zeta_t(dy,dw) \mu^\xi_t(dx,dv) \nu(d\zeta) \nu(d\xi)\right)^\alpha \\
&+\Delta_\phi\sigma_v(\mu_t)^{2\alpha-1}\left(\iint_{I^2\times\R^{4d}}\norm{v-w}^2\, \mu^\zeta_t(dy,dw) \mu^\xi_t(dx,dv) \nu(d\zeta) \nu(d\xi)\right)^{1/2}. \\
\end{aligned}
$$
It remains to note that 
$$
\sigma_v(\mu_t)^2=\dfrac{1}{2}\iint_{I^2\times\R^{4d}}\norm{v-w}^2\, \mu^\zeta_t(dy,dw) \mu^\xi_t(dx,dv) \nu(d\zeta) \nu(d\xi),
$$
to prove that
$$
\dfrac{d}{dt} \sigma_v(\mu_t)^2 \leq -(2^\alpha \Lambda_a \Lambda_K-\sqrt{2}\Delta_\phi)\sigma_v(\mu_t)^{2\alpha},
$$
which proves, using Assumption~\ref{ass:glob}.\ref{ass:cst} that $\sigma_v(\mu_t)\leq\sigma_v(\mu_0)\leq r_0$.
\end{proof}

\begin{lemma}\label{lem:estim_r_t}
Let $a$, $K$, $\phi$ \text{and $\nu$} satisfy {Assumptions~\ref{ass:loc} if $\alpha=1$ or Assumptions~\ref{ass:loc} and \ref{ass:glob} if $\alpha>1$}. Let $\mu\in C([0,T),\mathcal{P}_{\nu}(I \times \R^{2d}))$ be a local solution of \eqref{eq:KCD} such that a.e. $\xi \in I$, $\Supp(\mu_0^\xi)\subset B(0,r_0)$. Then for all $t\in[0,T]$, $\Supp(\mu_t)\subset B(0,r(t))$ where  
$$
r(t)\leq \left(r_0+\dfrac{C_2}{C_1}\right)e^{C_1t}-\dfrac{C_2}{C_1},
$$
where {$C_1$ and $C_2$ are explicit positive constants}.
\end{lemma}

\begin{proof}
First, we have
$$
r(t):=\esssup_{\xi\in I}\sup_{(x,v)\in\Supp(\mu_t)}\norm{x,v}=\esssup_{\xi\in I}\sup_{(x,v)\in\Supp(\mu_0^\xi)}\norm{T_{0,t}^\xi[a,\mu](x,v)},
$$
Then, let $\xi\in I$ such that $t\mapsto\mu^\xi_t$ is solution of \eqref{eq:distributional-solution-push-forward} and $r^\xi(t):=\sup_{(x,v)\in\Supp(\mu_0^\xi)}\norm{T_{0,t}^\xi[a,\mu](x,v)}\leq r(t)$. Then we have
$$
\partial_t\norm{T_{0,t}^\xi[a,\mu](x,v)}\leq\norm{H_a[\mu_t](\xi,T_{0,t}^\xi[a,\mu](x,v))}.
$$
where $H_a[\mu](\xi,x,v):=(v,F_a[\mu](x,v))$. {If $\alpha=1$ then using the fact that $\sigma_v(\mu_t)\leq r(t)$, we have 
$$
\norm{H_a[\mu_t](\xi,T_{0,t}^\xi[a,\mu](x,v))}\leq C_1r(t)+C_2.
$$
with $C_1 = 1+\Delta_\phi+4\norm{a}_\infty \Delta_{K}$ and $C_2 = \norm{a}_\infty \Delta_{K}$. If $\alpha>1$ then using Assumption~\ref{ass:loc} and Lemma~\ref{lem:dec_sigma}, we have 
$$
\norm{H_a[\mu_t](\xi,T_{0,t}^\xi[a,\mu](x,v))}\leq C_1r(t)+C_2,
$$
with $C_1 = 1+4\norm{a}_\infty \Delta_{K}$ and $C_2 = \norm{a}_\infty \Delta_{K}+\Delta_\phi r_0^{2\alpha-1}$.
In both cases,} it follows that
$$
r(t)\leq r(0)+\int_0^t(C_1r(s)+C_2)\,ds,
$$
which concludes the proof, using the Grönwall's lemma and Lemma~\ref{lem:estim_T_t} from which we prove the continuity of $r(t)$.
\end{proof}

\begin{proof}[Proof of Theorem \ref{theo:well-posedness-vlasov-global}]
From Lemma~\ref{lem:estim_r_t}, we directly have that the support of the solution does not blow up in finite time and thus that the maximal solution is global for all compactly supported $\mu_0$.
\end{proof}

\section{Stability estimates}\label{sec:stability}

In this section, we will study the stability of the Kinetic Cucker-Dong equation \eqref{eq:KCD} from which we will derive the mean-field limit of \eqref{eq:CD} to \eqref{eq:KCD}.

\begin{theorem}\label{thm:stability}
Let $a$, $K$, $\phi$ and $\nu$ satisfy { Assumptions~\ref{ass:loc} if $\alpha=1$ or Assumptions~\ref{ass:loc} and \ref{ass:glob} if $\alpha>1$. Let $\Tilde{a}\in L^\infty(I^2)$ also satisfies the same assumptions together with $K$ and $\phi$. Let us also assume $K$ satisfies Assumption~\ref{ass:integrability} and $\nu$  satisfies Assumption~\ref{ass:nu}.} Let $\mu, \Tilde \mu \in C(\R_+,\mathcal{P}_{2,\nu}(I \times \R^{2d}))$ be two global solutions of the Kinetic Cucker-Dong equation \eqref{eq:KCD} respectively related to $a$ and $\Tilde{a}$ and initial conditions $\mu_0, \Tilde \mu_0 \in \mathcal{P}_\nu(I \times \R^{2d})$ such that for a.e. $\xi\in I$, $\Supp(\mu_0^\xi),\Supp(\Tilde\mu_0^\xi)\subset B(0,r_0)$. Then we have for all $t\geq0$
$$
\W_{2,\nu}(\mu_t,\Tilde \mu_t)\leq C(t)\left(\W_{2,\nu}\left(\mu_0,\Tilde \mu_0\right) + \sqrt{d_\square(a,\Tilde{a})}\right),
$$
where $C(t)$ only depends on $\nu$, $\mu_0$, $\Tilde \mu_0$, $a$, $\Tilde{a}$, $d$, $\alpha$, $K$ and $\phi$.
\end{theorem}

\begin{proof}
From Lemma~\ref{lem:estim_r_t}, we have $\Supp(\mu_t),\Supp(\Tilde{\mu}_t)\subset B(0,r(t))$ where 
$$
r(t) := \left(r(0)+\dfrac{C_2}{C_1}\right)e^{C_1t}-\dfrac{C_2}{C_1}.
$$

Since $\W_{2,\nu}$ is a distance, we have
$$
\begin{aligned}
\W_{2,\nu}(\mu_t,\Tilde{\mu}_t)\leq&\left(\int_I\W_2\left(T^\xi_{0,t}[a,\mu]_{\#}\mu_0^\xi,T^\xi_{0,t}[\Tilde{a},\Tilde{\mu}]_{\#}\Tilde{\mu}_0^\xi\right)^2\nu(d\xi)\right)^{1/2} \\
\leq&\left(\int_I\W_2\left(T^\xi_{0,t}[a,\mu]_{\#}\mu_0^\xi,T^\xi_{0,t}[a,\mu]_{\#}\Tilde{\mu}_0^\xi\right)^2\nu(d\xi)\right)^{1/2} \\
&+\left(\int_I\W_2\left(T^\xi_{0,t}[a,\mu]_{\#}\Tilde{\mu}_0^\xi,T^\xi_{0,t}[\Tilde{a},\mu]_{\#}\Tilde{\mu}_0^\xi\right)^2\nu(d\xi)\right)^{1/2} \\
&+\left(\int_I\W_2\left(T^\xi_{0,t}[\Tilde{a},\mu]_{\#}\Tilde{\mu}_0^\xi,T^\xi_{0,t}[\Tilde{a},\Tilde{\mu}]_{\#}\Tilde{\mu}_0^\xi\right)^2\nu(d\xi)\right)^{1/2} \\
=&\,I_1+I_2+I_3.
\end{aligned}
$$

\paragraph{Control of $I_1$:} From a direct derivation and the Gronwall's lemma, we have
$$
\norm{T^\xi_{0,t}[a,\mu](x_1,v_1)-T^\xi_{0,t}[a,\mu](x_2,x_2)}\leq e^{L_{2,r(t)}t}\norm{(x_1,v_1)-(x_2,v_2)}.
$$

Then, using \cite[Lemma 3.13]{canizo2011well}, we have
$$
\W_2\left(T^\xi_{0,t}[a,\mu]_{\#}\mu_0^\xi,T^\xi_{0,t}[a,\mu]_{\#}\Tilde{\mu}_0^\xi\right) \leq e^{L_{2,r(t)}t}\W_2\left(\mu_0^\xi,\Tilde{\mu}_0^\xi\right),
$$
which leads to
$$
I_1\leq C_1(t)\W_{2,\nu}\left(\mu_0,\Tilde{\mu}_0\right)\quad\text{where}\quad C_1(t)=e^{L_{2,r(t)}t}.
$$

\paragraph{Control of $I_3$:} From \cite[Lemma 3.11]{canizo2011well}, we directly have that
$$
\W_2\left(T^\xi_{0,t}[\Tilde{a},\mu]_{\#}\Tilde{\mu}_0^\xi,T^\xi_{0,t}[\Tilde{a},\Tilde{\mu}]_{\#}\Tilde{\mu}_0^\xi\right) \leq \esssup_{(x_0,v_0)\in B(0,r_0)}\norm{T^\xi_{0,t}[\Tilde{a},\mu](x_0,v_0)-T^\xi_{0,t}[\Tilde{a},\Tilde{\mu}](x_0,v_0)}.
$$
Reasoning as in Theorem~\ref{theo:well-posedness-vlasov}, we obtain
$$
\begin{aligned}
\norm{T^\xi_{0,t}[\Tilde{a},\mu](x_0,v_0)-T^\xi_{0,t}[\Tilde{a},\Tilde{\mu}](x_0,v_0)}\leq& L_{2,r(t)}\int_0^t\norm{T^\xi_{0,s}[\Tilde{a},\mu](x_0,v_0)-T^\xi_{0,s}[\Tilde{a},\Tilde{\mu}](x_0,v_0)}\,ds \\
&+L_{1,r(t)}\int_0^t\W_{2,\nu}(\mu_s,\Tilde{\mu}_s)\,ds,
\end{aligned}
$$
which leads, using the Grönwall's Lemma and the fact that $L_{1,r(t)}$ can be chosen increasing, to 
$$
\W_2\left(T^\xi_{0,t}[\Tilde{a},\mu]_{\#}\Tilde{\mu}_0^\xi,T^\xi_{0,t}[\Tilde{a},\Tilde{\mu}]_{\#}\Tilde{\mu}_0^\xi\right) \leq L_{1,r(t)}e^{L_{2,r(t)}t}\int_0^t\W_{2,\nu}(\mu_s,\Tilde{\mu}_s)\,ds,
$$
and thus
$$
I_3\leq C_3(t)\int_0^t\W_{2,\nu}(\mu_s,\Tilde{\mu}_s)\,ds\quad\text{where}\quad C_3(t)=L_{1,r(t)}e^{L_{2,r(t)}t}.
$$

\paragraph{Control of $I_2$:} Considering the same transference plan as that in the proof of Lemma~\cite[Lemma 3.11]{canizo2011well} and using the definition of $\mathcal{W}_2$, we have 
$$
I_2^2\leq \iint_{I\times\R^{2d}}\norm{T_{0,t}^\xi[a,\mu](x,v)-T_{0,t}^\xi[\Tilde{a},\mu](x,v)}^2\Tilde{\mu}_0^\xi(dx,dv)\nu(d\xi).
$$
Since for all $(x,v)\in\Supp(\Tilde{\mu}_0)\subset B(0,r_0)$, we have $\norm{T_{0,t}^\xi[a,\mu](x,v)}\leq r(t)$ and $\norm{T_{0,t}^\xi[\Tilde{a},\mu](x,v)}\leq r(t)$, it follows that
$$
I_2^2\leq 2r(t)\iint_{I\times\R^{2d}}\norm{T_{0,t}^\xi[a,\mu](x,v)-T_{0,t}^\xi[\Tilde{a},\mu](x,v)}\Tilde{\mu}_0^\xi(dx,dv)\nu(d\xi).
$$
Then, we note that
$$
\begin{aligned}
&\iint_{I\times\R^{2d}}\norm{T_{0,t}^\xi[a,\mu](x,v)-T_{0,t}^\xi[\Tilde{a},\mu](x,v)}\,\Tilde{\mu}_0^\xi(dx,dv)\nu(d\xi) \\
\leq& \int_0^t\iint_{I\times\R^{2d}}\norm{H^\xi_a[\mu_s](T_{0,s}^\xi[a,\mu](x,v))-H^\xi_a[\mu_s](T_{0,s}^\xi[\Tilde{a},\mu](x,v))}\,\Tilde{\mu}_0^\xi(dx,dv)\nu(d\xi)ds \\
&+\int_0^t\iint_{I\times\R^{2d}}\norm{F^\xi_a[\mu_s](T_{0,s}^\xi[\Tilde{a},\mu](x,v))-F^\xi_{\Tilde{a}}[\mu_s](T_{0,s}^\xi[\Tilde{a},\mu](x,v))}\,\Tilde{\mu}_0^\xi(dx,dv)\nu(d\xi)ds. \\
\end{aligned}
$$

On the one hand, we have 
$$
\begin{aligned}
&\int_0^t\iint_{I\times\R^{2d}}\norm{H^\xi_a[\mu_s](T_{0,s}^\xi[a,\mu](x,v))-H^\xi_a[\mu_s](T_{0,s}^\xi[\Tilde{a},\mu](x,v))}\,\Tilde{\mu}_0^\xi(dx,dv)\nu(d\xi)ds \\
&\leq L_{2,r(t)}\int_0^t\iint_{I\times\R^{2d}}\norm{T_{0,s}^\xi[a,\mu](x,v)-T_{0,s}^\xi[\Tilde{a},\mu](x,v)}\,\Tilde{\mu}_0^\xi(dx,dv)\nu(d\xi)ds.
\end{aligned}
$$
On the other hand, setting $(X_s(x,v),V_s(x,v)):=T_{0,s}^\xi[\Tilde{a},\mu](x,v)$, we have
$$
\begin{aligned}
&\int_0^t\iint_{I\times\R^{2d}}\norm{F^\xi_a[\mu_s](T_{0,s}^\xi[\Tilde{a},\mu](x,v))-F^\xi_{\Tilde{a}}[\mu_s](T_{0,s}^\xi[\Tilde{a},\mu](x,v))}\,\Tilde{\mu}_0^\xi(dx,dv)\nu(d\xi)ds \\
&=\int_0^t\iint_{I\times\R^{2d}}\norm{\int_I(a(\xi,\zeta)-\Tilde{a}(\xi,\zeta))\left(\int_{\R^{2d}}K(y-X_s(x,v),w-V_s(x,v))\,\mu_s^\zeta(dy,dw)\right)\nu(d\zeta)}\,\Tilde{\mu}_0^\xi(dx,dv)\nu(d\xi)ds.
\end{aligned}
$$
From Assumption~\ref{ass:integrability}, we have 
$$
K(x,v)=\int_{\R^{2d}}\widehat{K}(\theta,\omega)e^{2\pi i(\theta\cdot x+ \omega\cdot v)}\,d\theta d\omega\quad\text{where}\quad\widehat{K}(\theta,\omega)=\int_{\R^{2d}}K(x,v)e^{-2\pi i(x\cdot \theta+ v\cdot \omega)}\,dxdv.
$$
Consequently, we have 
$$
\begin{aligned}
&\norm{\int_I(a(\xi,\zeta)-\Tilde{a}(\xi,\zeta))\left(\int_{\R^{2d}}K(y-X_s(x,v),w-V_s(x,v))\,\mu_s^\zeta(dy,dw)\right)\nu(d\zeta)} \\
&=\norm{\int_I(a(\xi,\zeta)-\Tilde{a}(\xi,\zeta))\left(\int_{\R^{2d}}\int_{\R^{2d}}\widehat{K}(\theta,\omega)e^{2\pi i(\theta\cdot (y-X_s(x,v))+ \omega\cdot (w-V_s(x,v)))}\,d\theta d\omega\,\mu_s^\zeta(dy,dw)\right)\nu(d\zeta)} \\
&=\norm{\int_{\R^{2d}}\int_I(a(\xi,\zeta)-\Tilde{a}(\xi,\zeta))\left(\int_{\R^{2d}}\widehat{K}(\theta,\omega)e^{2\pi i(\theta\cdot y+ \omega\cdot w)}e^{-2\pi i(\theta\cdot X_s(x,v)+ \omega\cdot V_s(x,v))}\,\mu_s^\zeta(dy,dw)\right)\nu(d\zeta)\,d\theta d\omega} \\
&\leq\int_{\R^{2d}}\norm{\widehat{K}(\theta,\omega)}\abs{\int_I(a(\xi,\zeta)-\Tilde{a}(\xi,\zeta))\Tilde g (s,\zeta,\theta,\omega)\,\nu(d\zeta)}\,d\theta d\omega,
\end{aligned}
$$
where
$$
\Tilde g (s,\zeta,\theta,\omega):=\int_{\R^{2d}}e^{2\pi i(\theta\cdot y+ \omega\cdot w)}\,\mu_s^\zeta(dy,dw),
$$
which leads to
$$
\begin{aligned}
&\int_0^t\iint_{I\times\R^{2d}}\norm{F^\xi_a[\mu_s](T_{0,s}^\xi[\Tilde{a},\mu](x,v))-F^\xi_{\Tilde{a}}[\mu_s](T_{0,s}^\xi[\Tilde{a},\mu](x,v))}\,\Tilde{\mu}_0^\xi(dx,dv)\nu(d\xi)ds \\
&\leq\int_0^t\int_{\R^{2d}}\norm{\widehat{K}(\theta,\omega)}\left(\int_I\abs{\int_I(a(\xi,\zeta)-\Tilde{a}(\xi,\zeta))\widehat{\mu}_s^\zeta(\theta,\omega)\,\nu(d\zeta)}\,\nu(d\xi)\right)d\theta d\omega ds.
\end{aligned}
$$
Now, since $\abs{\Tilde g (s,\zeta,\theta,\omega)}\leq1$, we note that 
$$
\begin{aligned}
&\int_I\abs{\int_I(a(\xi,\zeta)-\Tilde{a}(\xi,\zeta))\Tilde g (s,\zeta,\theta,\omega)\,\nu(d\zeta)}\,\nu(d\xi) \\
&\leq \int_I\abs{\int_I(a(\xi,\zeta)-\Tilde{a}(\xi,\zeta))\Re(\Tilde g (s,\zeta,\theta,\omega))\,\nu(d\zeta)}\,\nu(d\xi)+\int_I\abs{\int_I(a(\xi,\zeta)-\Tilde{a}(\xi,\zeta))\Im(\Tilde g (s,\zeta,\theta,\omega))\,\nu(d\zeta)}\,\nu(d\xi) \\
&\leq 2\sup_{\norm{g}_\infty\leq1}\int_I\abs{\int_I(a(\xi,\zeta)-\Tilde{a}(\xi,\zeta))g(\zeta)\,\nu(d\zeta)}\,\nu(d\xi)\leq  8\Delta_\nu^2d_\square(a,\Tilde{a}).
\end{aligned}
$$
where we used Proposition~\ref{prop:comparisonLinftyL1} for the last inequality. Finally, we have 
$$
\begin{aligned}
&\iint_{I\times\R^{2d}}\norm{T_{0,t}^\xi[a,\mu](x,v)-T_{0,t}^\xi[\Tilde{a},\mu](x,v)}\,\Tilde{\mu}_0^\xi(dx,dv)\nu(d\xi) \\
&\leq L_{2,r(t)}\int_0^t\iint_{I\times\R^{2d}}\norm{T_{0,s}^\xi[a,\mu](x,v)-T_{0,s}^\xi[\Tilde{a},\mu](x,v)}\,\Tilde{\mu}_0^\xi(dx,dv)\nu(d\xi)ds + 8\Delta_\nu^2 \bigl\|\widehat{K}\bigr\|_{L^1} d_\square(a,\Tilde{a}) t,
\end{aligned}
$$
and thus, from the Gronwall's lemma, 
$$
I_2\leq C_2(t)\sqrt{d_\square(a,\Tilde{a})}\quad\text{where}\quad C_2(t)=\sqrt{16\Delta_\nu^2\bigl\|\widehat{K}\bigr\|_{L^1}r(t)te^{L_{2,r(t)}t}}.
$$
\paragraph{Conclusion:} It follows that
$$
\W_{2,\nu}(\mu_t,\Tilde{\mu}_t)\leq C_1(t) \W_{2,\nu}\left(\mu_0,\Tilde{\mu}_0\right)+ C_2(t)\sqrt{d_\square(a,\Tilde{a})} + C_3(t)\int_0^t\W_{2,\nu}(\mu_s,\Tilde{\mu}_s)\,ds,
$$
which leads to 
$$
\W_{2,\nu}(\mu_t,\Tilde{\mu}_t)\leq C(t)\left(\W_{2,\nu}\left(\mu_0,\Tilde{\mu}_0\right) + \sqrt{d_\square(a,\Tilde{a})}\right)\quad\text{where}\quad C(t)=\max(C_1(t),C_2(t))e^{C_3(t)t}.
$$
\end{proof}

\begin{remark}
Note that Theorem~\ref{thm:stability} holds (for another constant $C_2(t)$) without Assumption~\ref{ass:integrability}, assuming that there exists $\Delta_\pi>0$ and $\pi_0\in\mathcal{P}(I)$ such that for almost all $\xi\in I$, 
\begin{equation}\label{eq:ass_integrability_bis}
    \Tilde{\mu}_0^\xi(dx,dv)\leq \Delta_\pi \pi_0(dx,dv).
\end{equation}
This comes from the fact that from Assumption~\ref{ass:loc}.\ref{ass:K}, we have 
$$
\norm{K(y-X_s(x,v),w-V_s(x,v))}\leq \Delta_K(1+4r(t)),
$$
and thus, from Proposition~\ref{prop:comparisonLinftyL1},
$$
\begin{aligned}
&\int_0^t\iint_{I\times\R^{2d}}\norm{\int_I(a(\xi,\zeta)-\Tilde{a}(\xi,\zeta))\left(\int_{\R^{2d}}K(y-X_s(x,v),w-V_s(x,v))\,\mu_s^\zeta(dy,dw)\right)\nu(d\zeta)}\,\Tilde{\mu}_0^\xi(dx,dv)\nu(d\xi)ds. \\
&\leq \int_0^t\iint_{I\times\R^{2d}}\norm{\int_I(a(\xi,\zeta)-\Tilde{a}(\xi,\zeta))\left(\int_{\R^{2d}}K(y-X_s(x,v),w-V_s(x,v))\,\mu_s^\zeta(dy,dw)\right)\nu(d\zeta)}_1\,\Tilde{\mu}_0^\xi(dx,dv)\nu(d\xi)ds \\
&\leq \sum_{l=1}^d \int_0^t\iint_{I\times\R^{2d}}\abs{\int_I(a(\xi,\zeta)-\Tilde{a}(\xi,\zeta))\left(\int_{\R^{2d}}K_l(y-X_s(x,v),w-V_s(x,v))\,\mu_s^\zeta(dy,dw)\right)\nu(d\zeta)}\,\Tilde{\mu}_0^\xi(dx,dv)\nu(d\xi)ds \\
&\leq \Delta_\pi\sum_{l=1}^d \int_0^t\int_{\R^{2d}}\int_I\abs{\int_I(a(\xi,\zeta)-\Tilde{a}(\xi,\zeta))\left(\int_{\R^{2d}}K_l(y-X_s(x,v),w-V_s(x,v))\,\mu_s^\zeta(dy,dw)\right)\nu(d\zeta)}\,\nu(d\xi)\pi_0(dx,dv)ds \\
&\leq 4d\Delta_K\Delta_\pi\Delta_\nu^2(1+4r(t))d_\square(a,\Tilde{a}) t.
\end{aligned}
$$
Contrary to Assumption~\ref{ass:integrability}, this result does not require any assumption on the regularity of $K$ but proves the stability for a smaller class of initial conditions. A sufficient condition for \eqref{eq:ass_integrability_bis} to be satisfy is that the densities of $\Tilde{\mu}_0^\xi$ are uniformly bounded from above. 
\end{remark}

\section{Mean-Field limit}\label{sec:mean-field}

{
\begin{proof}[Proof of Theorem \ref{theo:mfl}]
Let $a$, $K$ and $\phi$ satisfy Assumptions~\ref{ass:loc} if $\alpha=1$ or Assumptions~\ref{ass:loc} and \ref{ass:glob} if $\alpha>1$ and Assumption~\ref{ass:integrability}. Let $(x_i,v_i)_{i\in\intint{1}{N}}$ be the solution of \eqref{eq:CD} starting from $(x_i(0),v_i(0))_{i\in\intint{1}{N}}$. Let for all $t\geq0$, 

$$
\mu_t^{N,\xi}:= \sum_{i=1}^N \delta_{(x_i(t),v_i(t))}\1_i^N(\xi),
$$ 

and $\mu_t^N=\mu_t^{N,\xi}\otimes \lambda(d\xi)$ where $\lambda$ is the Lebesgue measure on $I$. We easily verify that $\mu_t^N\in C([0,T],\mathcal{P}_\nu(I \times \R^{2d}))$ and that for all $t\in[0,T)$, for all $\varphi\in C^1_c([0,T)\times I \times \R^{2d})$,

$$
\iint_{I\times\R^{2d}}\varphi(t,\xi,x,v)\mu_t^N(d\xi,dx,dv)=\int_I\sum_{i=1}^N\varphi(t,\xi,x_i,v_i)\1_i^N(\xi)\,d\xi.
$$
It follows, by a direct computation, that
$$
\begin{aligned}
&\partial_t\left(\iint_{I\times\R^{2d}}\varphi(t,\xi,x,v)\mu_t^N(d\xi,dx,dv)\right) \\
&=\iint_{I\times\R^{2d}}\partial_t\varphi(t,\xi,x,v)+v\cdot\grad_x\varphi(t,\xi,x,v)+F_a^\xi[\mu_t^N](x,v)\cdot\grad_v\varphi(t,\xi,x,v)\,\mu_t^N(d\xi,dx,dv). \\
\end{aligned}
$$
which proves that $\mu_t^N$ is a solution of \eqref{eq:KCD} in the sense of Definition~\ref{defi:distributional-solution-Vlasov-joint}. Since the Lebesgue measure satisfy Assumption~\ref{ass:nu} it follows from Theorem~\ref{thm:stability} that for all other solution $\mu_t$ related to a kernel $a$ that satisfies Assumptions~\ref{ass:loc} if $\alpha=1$ or Assumptions~\ref{ass:loc} and \ref{ass:glob} if $\alpha>1$ and Assumption~\ref{ass:integrability} together with $K$ and $\phi$, we have 
$$
\W_{2,\nu}(\mu_t^N,\mu_t)\leq C(t)\left(\W_{2,\nu}\left(\mu_0^N,\mu_0\right) + \sqrt{d_\square(a_N,a)}\right),
$$
which leads to the expected results.
\end{proof}
}

\section{Conclusion}\label{sec:conclusion}

{We have therefore extended the rigorous derivation of the mean-field limit along two novel directions: first, by treating the case where repulsive forces are present, and second, by allowing for non-exchangeable particles. Concerning the first point, a first attempt to derive the mean-field limit for the Cucker–Dong model was proposed in \cite{yang2014} using the Wasserstein distance of order $1$. However, the proof of the main result claimed therein does not hold in general. More precisely, we show below that it is impossible to obtain an analogue of Lemma~\ref{lem:estim_F_mu} with $\W_1$ and any $\alpha< 3/2$.} Let $d=1$, $K=0$, 
$$
\mu(d\xi,dx,dv)=\chi^{\xi,x}(dv)\otimes \rho^\xi(dx)\otimes \nu(d\xi)\quad\text{and}\quad\mu_n(d\xi,dx,dv)=\chi_n^{\xi,x}(dv)\otimes \rho^\xi(dx)\otimes \nu(d\xi),
$$ 
where $\chi^{\xi,x}=\delta_0$ and $\chi_n^{\xi,x}=(1-1/n)\delta_0+\delta_1/n$. On the one hand we have 
$$
\begin{aligned}
&\norm{F_a[\mu_n](\xi,x,v) - F_a[\mu](\xi,x,v)}= C(x)\abs{\sigma_v(\mu_n)^{2\alpha-1}-\sigma_v(\mu)^{2\alpha-1}}, \\
&C(x)=\norm{\iint_{I\times \R}\phi(\norm{x-y})\,(x-y)\rho^\xi(dy)\nu(d\xi)}.
\end{aligned}
$$
Since we have 
$$
\sigma_v(\mu)^2=0\quad\text{and}\quad\sigma_v(\mu_n)^2=\frac{1}{n}-\frac{1}{n^2},
$$
it follows that 
$$
\norm{F_a[\mu_n](\xi,x,v) - F_a[\mu](\xi,x,v)}= C(x)\left(\frac{1}{n}-\frac{1}{n^2}\right)^{\alpha-1/2}.
$$
On the other hand, from \cite[Equation (3.13)]{peszek2023heterogeneous}, we have 
$$
\W_{1,\nu}(\mu_n,\mu)=\int_I\W_1(\mu_n^\xi,\mu^\xi)\,\nu(d\xi)\leq\iint_{I\times\R}\W_1(\chi_n^{\xi,x},\chi^{\xi,x})\,\rho^\xi(dx)\nu(d\xi)=\frac{1}{n}.
$$
It follows that if $\alpha<3/2$, for all $C>0$, there exists an $N>0$ such that for all $n\geq N$,
$$
\norm{F_a[\mu_n](\xi,x,v) - F_a[\mu](\xi,x,v)}>C\W_{1,\nu}(\mu_n,\mu).
$$

This difficulty motivates our use of the Wasserstein distance of order $2$ instead of the Wasserstein  distance of order $1$. Moreover, this choice is consistent with the model, in particular with the presence of the second-order velocity term $\sigma_v(\mu)$ appearing  in equation \eqref{eq:KCD1}. \\

This work is part of a broader project focused on deriving mean-field limits in the context of collision avoidance.  Achieving such regime requires repulsive forces that diverge rapidly in the neighborhood of particles. Unfortunately, our current analytical framework does not yet allow us to treat this setting. An alternative approach explored in the literature is to consider highly singular interaction forces. In this case, a cut-off is typically introduced to recover local Lipschitz regularity. Establishing the mean-field limit for particle systems subjected to highly singular forces  (ensuring collision avoidance)  with cut-off    in the non-exchangeable setting has not yet been addressed and thus constitutes a natural extension of the present work.

\section*{Acknowledgment}

The work of the authors is supported by the Agence Nationale de la Recherche through the grants ANR-23-CE40-0003 (Conviviality).

\bibliographystyle{abbrv}
\bibliography{refs.bib}
\end{document}